\documentclass[12pt,a4paper]{article}

\usepackage{hyperref}
\usepackage[english]{babel} 
\usepackage{amsthm}
\usepackage{amsmath}
\usepackage{amsfonts}
\usepackage{amssymb}
\usepackage{enumerate}
\usepackage{amscd}
\usepackage{graphicx}
\usepackage{makeidx}
\usepackage[all]{xy}

\makeindex

\newcounter{Def}[section]
\renewcommand{\theDef}{\arabic{section}.\arabic{Def}}

\theoremstyle{plain}
\newtheorem{Proposition}[Def]{Proposition}

\newtheorem{Example}[Def]{Example}

\newtheorem{Remark}[Def]{Remark}
\newtheorem{Warning}[Def]{Warning}
\newtheorem{Lemma}[Def]{Lemma}
\newtheorem{Theorem}[Def] {Theorem}

\newtheorem{corollary}[Def]{Corollary}
\newtheorem{Definition}[Def]{Definition}

\def\bigset#1#2{\left\lbrace\;\begin{minipage}[c]{#1}\begin{center}#2\end{center}\end{minipage}\;\right\rbrace}

\def\BrPic          {\underline{\text{Pic}}}

\def\Aut              {{\rm Aut}}
\def\AUT              {{\rm AUT}}
 
\def\R          {\mathbb R}

\def\Z          {\mathbb Z}

\def\cala          {\mathcal A}
\def\calb          {{\mathcal B}}
\def\calc          {{\mathcal C}}
\def\cald          {{\mathcal D}}

\def\call          {{\mathcal L}}

\def\calv          {{\mathcal V}}

\def\cyl           {{\Sigma(1,1)}}
\def\be                 {\begin{equation}}
\def\ee                 {\end{equation}}
\def\Desc          {\mathcal{D}\hspace{-.8pt}\mbox{\small\sl esc}}

\def\pop            {{\Sigma(2,1)}}

\def\KK           {{\mathbb K}}

\def\id               {{\rm id}}

\def\SS            {\mathbb{S}^1}

\def\id                 {\mathrm{id}}

\def\iso                {\stackrel{\sim}{\longrightarrow}}

  \def\Hom{\textnormal{Hom}}

\def\Ue            {\mathrm U(1)}

\def\Bun           {\mathcal{B}{\!un}}

\def\vect          {\textnormal{Vect}}

 \def\H              {H}
\def\End            {\text{End}}
\def\cob           {\mathfrak{C}\mathrm{ob}}
\def\Span            {\mathfrak{S}\mathrm{pan}}
\def\mod            {\text{-}\mathrm{mod}}
\def\Inn            {\mathrm{Inn}}
\def\Out            {\mathrm{Out}}
\def\h                          {\check{\mathrm{H}}}
\def\cgj                        {\calc^J(G) = \bigoplus_{j \in J} \mathcal{C}(G)_j}
\def\cdw            {\calc^J(G)}

\begin{document}

\numberwithin{equation}{section}


\thispagestyle{empty}
\begin{flushright}
    {\sf ZMP-HH/11-3}\\
    {\sf Hamburger$\;$Beitr\"age$\;$zur$\;$Mathematik$\;$Nr.$\;$402}\\[2mm]
   March 2011
\end{flushright}
\vskip 2.0em
\begin{center}\Large
Equivariant Modular Categories via Dijkgraaf-Witten Theory
\end{center}\vskip 1.4em
\begin{center}
Jennifer Maier $^{1,2}$, Thomas Nikolaus $^2$ and Christoph Schweigert $^2$
\end{center}

\vskip 3mm

\begin{center}\it
   $^1$ School of Mathematics, Cardiff University, \\
   Senghennydd Road, Cardiff CF24 4AG, Wales, UK \\[.8em]
   $^2$ Fachbereich Mathematik, \ Universit\"at Hamburg\\
   Bereich Algebra und Zahlentheorie\\
   Bundesstra\ss e 55, \ D\,--\,20\,146\, Hamburg  
\end{center}
\vskip 2.5em
\begin{abstract} \noindent
Based on a weak action of a finite group $J$ on a finite group $G$,
we present a geometric construction of $J$-equivariant Dijkgraaf-Witten theory
as an extended topological field theory. The construction yields an
explicitly accessible class of equivariant modular tensor categories.
For the action of a group $J$ on a group $G$, the
category is described as the representation category of a
$J$-ribbon algebra that generalizes the Drinfel'd double
of the finite group $G$.

\end{abstract}

\noindent
{\sc Keywords}: Extended topological field theory, Dijkgraaf-Witten theory,
weak group action, equivariant modular tensor category, equivariant
ribbon algebras.

\setcounter{footnote}{0} \def\thefootnote{\arabic{footnote}}

\tableofcontents

\section{Introduction}\label{sec:introduction}

This paper has two seemingly different motivations and,
correspondingly, can be read from two different points of view,
a more algebraic and a more geometric one. Both in the introduction and
the main body of the paper, we try to separate these two points of view 
as much as possible, in the hope to keep the paper accessible
for readers with specific interests.

\subsection{Algebraic motivation: equivariant modular categories}
\label{intro:1.1}

Among tensor categories, modular tensor categories
are of particular interest for representation theory and
mathematical physics.
The representation categories of several algebraic
structures give examples of semisimple modular tensor categories:
\\[-1.66em]\begin{enumerate}\addtolength\itemsep{-2pt}
\item 
Left modules over connected factorizable ribbon weak Hopf algebras 
with Haar integral over an algebraically closed field {\rm \cite{nitv}}.
\item 
Local sectors of a finite $\mu$-index net of von Neumann algebras
on $\R$, if the net is strongly additive and split {\rm \cite{KLM}}.
\item 
Representations of selfdual $C_2$-cofinite vertex algebras 
with an additional finiteness condition on the homogeneous
components and which have semisimple representation categories {\rm \cite{Hu}}.
\end{enumerate}

Despite this list and the rather different fields in which modular
tensor categories arise, it is fair to say that modular
tensor categories are rare mathematical objects.
Arguably, the simplest incarnation of the first
algebraic structure in the list is the Drinfel'd double $\cald(G)$
of a finite group $G$.
Bantay \cite{bantay} has suggested a more general source for
modular tensor categories: a pair, consisting of a finite group $H$ and
a normal subgroup  $G\triangleleft H$. (In fact, Bantay has suggested general
finite crossed modules, but for this paper, only the case of a normal
subgroup is relevant.) In this situation, Bantay constructs a
ribbon category which is, in a natural
way, a representation category of a ribbon Hopf algebra
$\calb(G\triangleleft H)$. Unfortunately, it turns out that,
for a proper subgroup inclusion,
the category $\calb(G\triangleleft H)\mod$ is only premodular and 
not modular.

Still, the category $\calb(G\triangleleft H)\mod$ is modularizable in the
sense of Brugui\`eres \cite{brug}, and the next 
candidate for new modular tensor categories is the modularization
of $\calb(G\triangleleft H)\mod$.
However, it has been shown \cite{maier} that this modularization
is equivalent to the representation category of
the Drinfel'd double $\cald(G)$.

The modularization procedure of Brugui\`eres is based
on the observation that the violation of modularity of a
modularizable tensor category $\calc$ is captured in terms
of a canonical Tannakian subcategory of $\calc$. For
the category $\calb(G\triangleleft H)\mod$, this subcategory can
be realized as the representation category of the 
the quotient group $J:= H/G$ \cite{maier}. The modularization functor 
$$ \calb(G\triangleleft H)\mod \to \cald(G)\mod $$
is induction along the commutative Frobenius algebra given by the
regular representation of $J$. This has the important
consequence that the modularized category $\cald(G)$
is endowed with a $J$-action.

Experience with orbifold constructions, see \cite{kirI17,turaev2010}
for a categorical formulation, raises the question
of whether the category $\cald(G)\mod$ with this $J$-action
can be seen in a natural way as
the neutral sector of a $J$-modular tensor category. 

We thus want to complete the following square of tensor categories

\begin{equation}\label{int:square}
\xymatrix{
\cald(G)\mod \ar@<0.7ex>[d]^{\text{orbifold}} \ar@(lu,ur)[]^J\ar@{^{(}->}[r] 
& ??? \ar@(lu,ur)[]^J \ar@<0.7ex>[d]^{\text{orbifold}} \\
\calb(G\triangleleft H)\mod \ar[u]^{\text{modularization}}\ar@{^{(}->}[r] 
& ??? \ar[u]
}
\end{equation}

Here vertical arrows pointing upwards stand for
induction functors along the commutative
algebra given by the regular representation of $J$,
while downwards pointing arrows indicate orbifoldization.
In the upper right corner, we wish to place a $J$-modular category,
and in the lower right corner its $J$-orbifold which, on general
grounds \cite{kirI17}, has to be a modular tensor category.
Horizontal arrows indicate the inclusion of neutral sectors.

In general, such a completion need not exist. Even if it exists,
there might be inequivalent choices of $J$-modular tensor categories
of which a given modular tensor category with $J$-action is the neutral
sector \cite{ENO2009}.

\subsection{Geometric motivation: equivariant extended TFT}

Topological field theory is a mathematical structure that has been inspired
by physical theories \cite{witten} and which has developed
into an important tool in low-dimensional topology. Recently, these theories have received increased attention due
to the advent of {\em extended} topological field theories
\cite{Lurie,priesPhD}. The present paper focuses on three-dimensional 
topological field theory. 

Dijkgraaf-Witten theories provide a class of extended topological 
field theories. They can be seen as discrete variants of Chern-Simons 
theories, which provide invariants of three-manifolds and 
play an important role in knot theory \cite{witten}. 
Dijkgraaf-Witten theories have the advantage of being 
particularly tractable and admitting a very conceptual geometric construction. 

A Dijkgraaf-Witten theory is based on a finite group $G$; in this case the 
'field configurations' on a manifold $M$ are given by $G$-bundles over $M$, 
denoted by $\cala_G(M)$. Furthermore, one has to choose a suitable
action functional $S: \cala_G(M) \to \mathbb{C}$  
(which we choose here in fact to be trivial) on field configurations; 
this allows to make the structure suggested by formal path integration 
rigorous and to obtain a topological field theory. A  conceptually
very clear way 
to carry this construction out rigorously is described in \cite{FQ93} and 
\cite{Morton}, see section \ref{sec:DW} of this paper for a review.

Let us now assume that as a further input datum we have another finite group 
$J$ which acts on $G$. In this situation, we get an action of $J$ on the 
Dijkgraaf-Witten theory based on $G$. But it turns out that this
topological field 
theory  together with the $J$-action does not fully reflect 
the equivariance of the situation: it has been an important insight that 
the right notion is the one of equivariant topological field theories, 
which have been another point of recent interest \cite{kirI17,turaev2010}. 
Roughly speaking, equivariant topological field theories require that 
all geometric objects 
(i.e.\ manifolds of different dimensions) have to be decorated by a 
$J$-cover (see definitions \ref{defcobj} and \ref{defjtft} for details). 
Equivariant field theories also provide a conceptual setting for the 
orbifold construction, one of the standard tools for model building in 
conformal field theory and string theory. 

Given the action of a finite group $J$ on a finite group $G$, these
considerations lead to the question of whether Dijkgraaf-Witten theory 
based on $G$ can be enlarged to a $J$-equivariant topological field theory. 
Let us pose this question more in detail:
\begin{itemize}
\item What exactly is the right notion of an action of $J$ on $G$ 
that leads to interesting theories? To keep equivariant Dijkgraaf-Witten 
theory as explicit as the non-equivariant theory, one needs notions to keep 
control of this action as explicitly as possible.
\item 
Ordinary Dijkgraaf-Witten theory is mainly determined by the choice of 
field configurations $\cala_G(M)$ to be $G$-bundles. As mentioned before, 
for $J$-equivariant theories, we should replace manifolds
by manifolds with $J$-covers. We thus need a geometric notion of a 
$G$-bundle that is 'twisted' by this $J$-cover in order to develop the 
theory parallel to the non-equivariant one.
\end{itemize}
Based on an answer to these two points, we wish to construct
equivariant Dijkgraaf-Witten theory as explicitly as possible.

\subsection{Summary of the results}

This paper solves both the algebraic and the geometric problem we have 
just described. In fact, the two problems turn out to be closely related. 
We first solve the problem of explicitly constructing equivariant 
Dijkgraaf-Witten and then use our solution to construct the relevant 
modular categories that complete the square \eqref{int:square}. 

Despite this strong mathematical interrelation, we have taken some
effort to write the paper in such a way that it is accessible to
readers sharing only a geometric or algebraic interest. The geometrically
minded reader might wish to restrict his attention to section 2 and 3, 
and only take notice of the result about $J$-modularity
stated in theorem \ref{J-modular}. An algebraically oriented
reader, on the other hand, might simply accept the categories
described in proposition \ref{sectors} together with the structure
described in propositions \ref{prop:fusionproduct}, 
\ref{prop:action}
and \ref{prop:braiding} and then directly delve into section 4.

\medskip

For the benefit of all readers, we present here an outline of all our findings.
In section 2, we review the pertinent aspects of Dijkgraaf-Witten theory and 
in particular the specific construction given in \cite{Morton}. 
Section 3 is devoted to the equivariant case: we observe that the correct
notion of $J$-action on $G$ is what we call a weak
action of the group $J$ on the group $G$; this notion is 
introduced in definition \ref{Def:action}. Based on this notion,
we can very explicitly construct for every $J$-cover $P\to M$ a category 
$\cala_G(P \to M)$ of $P$-twisted $G$-bundles. For the definition and 
elementary properties of twisted bundles, we refer to section 
\ref{sec:twisted} and for a local description to appendix
\ref{cech}. We are then ready to construct equivariant Dijkgraaf
Witten theory along the lines of the construction described in
\cite{Morton}. This is carried out in section \ref{subsec:eqDW} and 
\ref{subsec:spans}. We obtain a construction of equivariant Dijkgraaf-Witten
theory that is so explicit that we can read off the category $\calc^J(G)$ it 
assigns to the circle $\SS$. The equivariant topological field 
theory induces additional structure on this category, which can also be 
computed by geometric methods due to the explicit control of the theory, 
and part of which we compute in section \ref{twisted_sectors_and_fusion}. 
This finishes the geometric part of our work. It remains to show
that the category $\calc^J(G)$ is indeed $J$-modular.

\medskip

To establish the $J$-modularity of the category
$\calc^J(G)$, we have to resort to algebraic tools. 
Our discussion is based on the appendix 6 of \cite{turaev2010}
by A.\ Vir\'elizier. At the same time,
we explain the solution of the algebraic problems described in
section \ref{intro:1.1}. The Hopf algebraic notions we encounter
in section 4, in particular Hopf algebras with a weak group 
action and their
orbifold Hopf algebras might be of independent algebraic interest.

In section 4, we introduce the notion of a $J$-equivariant ribbon
Hopf algebra. It turns out that it is natural to relax
some strictness requirements on the $J$-action on such a Hopf algebra.
Given a weak action of a finite group $J$ on a
finite group $G$, we describe in proposition \ref{double-J-ribbon} a specific
ribbon Hopf algebra which we call the equivariant Drinfel'd double $\cald^J(G)$.
This ribbon Hopf algebra is designed in such a way that its representation
category is equivalent to the geometric category $\calc^J(G)$
constructed in section 3, compare
proposition \ref{J-fusion}.

The $J$-modularity of $\calc^J(G)$ is established via the modularity of
its orbifold category. The corresponding notion of an orbifold algebra
is introduced in subsection 4.4.  In the case of the equivariant
Drinfel'd double $\cald^J(G)$, this orbifold algebra is shown to be
isomorphic, as a ribbon Hopf algebra, to a Drinfel'd double. This implies
modularity of the orbifold theory and, by a result of \cite{kirI17},
$J$-modularity of the category $\calc^J(G)$, cf.\ theorem \ref{J-modular}.

In the course of our construction, we develop 
several notions of independent interest.
In fact, our paper might be seen as a study of the
geometry of chiral backgrounds. It allows for various
generalizations, some of which are briefly sketched in the
conclusions. These generalizations include in
particular twists by 3-cocycles in group cohomology
and, possibly, even the case of non-semi simple chiral
backgrounds.

\bigskip

\noindent{\bf Acknowledgements.}
We thank Urs Schreiber for helpful discussions and Ingo
Runkel for a careful reading of the manuscript.
TN and CS are partially supported by the Collaborative 
Research Centre 676 
``Particles, Strings and the Early Universe - the Structure of Matter 
and Space-Time'' and the cluster of excellence
``Connecting particles with the cosmos''.
JM and CS are partially supported by the
Research priority program SPP 1388 
``Representation theory''.
JM is partially supported by the Marie Curie Research Training Network
MRTN-CT-2006-031962 in Noncommutative Geometry, EU-NCG.

\section{Dijkgraaf-Witten theory and Drinfel'd double}\label{sec:DW}

This section contains a short review of  Dijkgraaf-Witten 
theory as an extended three-dimensional topological
field theory, covering the contributions of many authors,
including in particular the work of Dijkgraaf-Witten 
\cite{DW90}, of Freed-Quinn \cite{FQ93}
and of Morton \cite{Morton}. We explain how these extended
3d TFTs give rise to modular tensor categories.
These specific modular tensor categories are the representation
categories of a well-known class of quantum groups, the
Drinfel'd doubles of finite groups.

While this section does not contain original material,
we present the ideas in such a way that equivariant generalizations
of the theories can be conveniently discussed. In this section, we also
introduce some categories and functors that we need for later sections.

\subsection{Motivation for  Dijkgraaf-Witten theory}
\label{sec:motivation}

We start with a brief motivation for Dijkgraaf-Witten 
theory from physical principles. A reader already familiar
with Dijkgraaf-Witten theory might wish to take
at least notice
of the definition \ref{def:2.2} and of proposition \ref{DW}.

It is an old, yet successful idea to extract invariants of 
manifolds from quantum field theories, in particular
from quantum field theories for which the fields are $G$-bundles 
with connection, where  $G$ is some group. In this paper we 
mostly consider the case of a finite group and only  occasionally 
make reference to the case of a compact Lie group.

Let $M$ be a compact oriented manifold of dimension 1,2 or 3, possibly 
with boundary. As the \textit{`space' of field configurations},
we choose $G$ bundles with connection,
$$ \mathcal A_G(M) := \Bun^{\nabla}_G(M). $$
In this way, we really assign to a manifold a groupoid, 
rather than an actual space.
The morphisms of the category take gauge transformations
into account.
We will nevertheless keep on calling it 'space' since 
the correct framework to handle $\mathcal A_G(M)$ is 
as a stack on  the category of smooth manifolds.

Moreover, another piece of data specifying the model 
is a function
defined on manifolds of a specific dimension,
$$ S: \mathcal A_G(M) \to \mathbb{C} $$
called the \textit{action}. In the simplest case, when $G$ is a finite group,
a field configuration is given by a $G$-bundle, since all bundles are
canonically flat and no connection data are involved.
Then, the simplest action is given by $S[P] := 0$ for all $P$. 
In the case of a compact, simple, simply connected Lie group $G$,
consider a 3-manifold $M$. In this situation, each $G$-bundle $P$ over $M$ 
is globally of the form $P \cong G \times M$, because 
$\pi_1(G) = \pi_2(G) = 0$. Hence a field configuration is given by a connection on the 
trivial bundle which is a 1-form $A \in \Omega^1(M,\mathfrak g)$ with values 
in the Lie algebra of $G$. An example of an action yielding
a topological field theory that can be defined in this 
situation is the Chern-Simons action
$$ S[A] := \int_M \langle A \wedge dA \rangle 
- \frac 1 6 \langle A \wedge A \wedge A \rangle $$ 
where $\langle \cdot , \cdot \rangle$ is the basic invariant inner product 
on the Lie algebra $\mathfrak g$. 

The heuristic idea is then to introduce an 
invariant $Z(M)$ for a 3-manifold $M$ by integration over all field
configurations: 
$$ Z(M) := " \int_{\mathcal A_G(M)} d\phi  ~ e^{i S[\phi]} ~".   $$

\begin{Warning} In general, this path integral has only a heuristic
meaning. In the case of a finite group, however, one can choose a 
counting measure $d\phi$ and thereby reduce the integral to a well-defined 
finite sum. The definition of Dijkgraaf-Witten theory \cite{DW90}
is based on this idea.
\end{Warning}

Instead of giving a well-defined meaning to the invariant $Z(M)$ as
a path-integral, we exhibit some formal properties these
invariants are expected to satisfy. To this end, it is crucial to allow
for manifolds that are not closed, as well. This allows to cut a 
three-manifold into several simpler three-manifolds with boundaries
so that the computation of the invariant can be reduced to the computation 
of the invariants of simpler pieces.

Hence, we consider a 3-manifold $M$ with a 2-dimensional boundary 
$\partial M$. We fix boundary values $\phi_1 \in \mathcal A_G(\partial M)$ 
and consider the space $\mathcal A_G(M,\phi_1)$ of all fields $\phi$ on $M$ 
that restrict to  the given boundary values $\phi_1$. We then introduce, again
at a heuristic level, the quantity
\begin{equation}\label{vektor}
  Z(M)_{\phi_1} := " \int_{\mathcal A_G(M,\phi_1)} d\phi  ~ e^{i S[\phi]} ~". \end{equation}
The assignment $\phi_1 \mapsto Z(M)_{\phi_1}$ could be called a
`wave function' on the space $\mathcal A_G(\partial M)$ 
of boundary values of fields.
These `wave functions' form a vector space 
$\mathcal H_{\partial M}$, the \textit{state space}
$$ \mathcal H_{\partial M} := "L^2 \big(\mathcal A_G(\partial M),\mathbb{C}\big)~" $$
that we assign to the boundary $\partial M$.  The transition to wave functions amounts to a
linearization. The notation $L^2$ should be taken with a grain
of salt and should indicate the choice of an appropriate
vector space for the category $\mathcal A_G(\partial M)$; 
it should not suggest the 
existence of any distinguished measure on the category.

In the case of Dijkgraaf-Witten theory based on a finite group $G$, 
the space of states has a basis consisting 
of $\delta$-functions on the set of isomorphism classes of 
field configurations on the boundary $\partial M$: 
$$ \mathcal H_{\partial M} = \mathbb{C} \big <\delta_{\phi_1} \mid \phi_1 \in Iso \mathcal A_G(\partial M) \big>.$$
In this way, we associate finite dimensional vector spaces 
$\mathcal H_{\Sigma}$ to compact oriented 2-manifolds 
$\Sigma$. The heuristic path integral in equation 
\eqref{vektor} suggests to associate to a
3-manifold $M$ with boundary $\partial M$
an element
$$ Z(M) \in \mathcal H_{\partial M}  \,\, , $$
or, equivalently, a linear map $\mathbb{C} \to  
\mathcal H_{\partial M}$. 

A natural generalization of this situation are
cobordisms $M:
\Sigma \to \Sigma'$, where $\Sigma$ and $\Sigma'$ are 
compact oriented 
2-manifolds. A cobordism is a compact oriented 3-manifold $M$ with 
boundary $\partial M \cong \bar{\Sigma} \sqcup \Sigma'$ 
where $\bar \Sigma$ 
denotes $\Sigma$, with the opposite orientation. 
To a cobordism, we wish to associate a linear map
$$ Z(M) : \mathcal H_{\Sigma} \to \mathcal H_{\Sigma'} $$
by giving its matrix elements in terms of the path 
integral
$$ Z(M)_{\phi_0,\phi_1} := " \int_{\mathcal A_G(M,\phi_0,\phi_1)} d\phi 
~ e^{i S[\phi]} ~" \,\,\, $$
with fixed boundary values
$\phi_0 \in \mathcal A_G(\Sigma)$ and 
$\phi_1 \in \mathcal A_G(\Sigma')$.
Here $\mathcal A_G(M,\phi_0,\phi_1)$ is the space of field 
configurations on $M$ that 
restrict to the field configuration $\phi_0$ on the ingoing 
boundary $\Sigma$ and to the field configuration $\phi_1$ 
on the outgoing boundary $\Sigma'$. One can now show  that the linear maps $Z(M)$ 
are compatible with gluing of cobordisms along boundaries. 
(If the group $G$ is not finite, additional subtleties arise; e.g.\
$Z(M)_{\phi_0,\phi_1}$ has to be interpreted as an integral kernel.)

Atiyah \cite{At88} has given a definition of 
a topological field theory that formalizes these
properties: it describes a topological field theory as
a symmetric monoidal functor from a 
geometric tensor category to an algebraic category. 
To make this definition explicit, let $\cob(2,3)$ be the category which has 2-dimensional compact 
oriented smooth manifolds as objects. Its morphisms
$M:\Sigma\to \Sigma'$ are given by (orientation preserving) 
diffeomorphism 
classes of 3-dimensional, compact oriented cobordism 
from $\Sigma$ to $\Sigma'$ which we write as
$$ \Sigma\hookrightarrow M \hookleftarrow \Sigma'.$$
Composition of morphisms is given by gluing cobordisms together along the 
boundary. The disjoint union of 2-dimensional manifolds and cobordisms equips this category 
with the structure of a symmetric monoidal category. For the
algebraic category, we choose the symmetric tensor category 
$\vect_{\KK}$ of finite dimensional vector spaces over an algebraically 
closed field $\KK$ of characteristic zero.

\begin{Definition}[Atiyah] \label{def:2.2}
A \textit{3d TFT} is a symmetric monoidal functor
$$ Z: \cob(2,3) \to \vect_{\KK}. $$
\end{Definition}

\medskip

Let us set up such a functor for Dijkgraaf-Witten theory,
i.e.\ fix a finite group $G$ and choose the 
trivial action $S: A_G(M) \to \mathbb{C}$,  i.e. $S[P] = 0$ 
for all $G$-bundles $P$ on $M$. 
Then the path integrals reduce to finite sums over $1$ 
hence simply count 
the number of elements in the category $\mathcal A_G$.
Since we are counting objects in a category, the
stabilizers have to be taken
appropriately into account, for details see e.g.\ 
\cite[Section 4]{Morton08}. This is achieved by
the  \textit{groupoid 
cardinality} (which is sometimes also called the 
Euler-characteristic of the groupoid $\Gamma$) 
$$ |\Gamma| := \sum_{[g] \in Iso(\Gamma)} \frac 1 {|Aut(g)|}. $$
A detailed discussion of groupoid cardinality can be found in \cite{BD00} 
and \cite{Leinster08}.

We summarize the discussion:

\begin{Proposition}[\cite{DW90},\cite{FQ93}] \label{DW}
Given a finite group $G$, the following assignment $Z_G$ defines
a 3d TFT: to a closed, oriented 2-manifold $\Sigma$, we assign 
the vector space freely generated by the isomorphism
classes of $G$-bundles on $\Sigma$, 
$$ \Sigma \quad \longmapsto \quad \mathcal H_\Sigma 
:= \KK \big< \delta_P \mid P \in 
Iso \mathcal A_G(\Sigma) \big> \,\, . $$
To a 3 dimensional cobordism $M$, we associate  the linear map
$$ Z_G\Big(\Sigma\hookrightarrow M \hookleftarrow \Sigma'\Big) : \quad \mathcal H_\Sigma \to \mathcal H_{\Sigma'}$$
with matrix elements given by the groupoid cardinality of
the categories $\cala_G(M,P_0,P_1)$:
$$  Z_G(M)_{P_0,P_1} := \big|\mathcal A_G(M,P_0,P_1)\big| \,\,\, . $$
\end{Proposition}

\begin{Remark}
\begin{enumerate}
\item
In the original paper \cite{DW90}, a generalization of the trivial action 
$S[P] = 0$, induced by an element $\eta$ in the group cohomology 
$\H^3_{Gp}\big(G,\Ue\big)$ with values in $\Ue$, has been studied. We postpone
the treatment of this generalization to a separate paper: in the present paper,
the term \textit{Dijkgraaf-Witten theory} refers to the 3d TFT 
of proposition \ref{DW} or its extended version. 
\item
In the case of a compact, simple, simply-connected Lie group $G$,
a definition of a 3d TFT by a path integral is not available. Instead,
the combinatorial definition of Reshetikin-Turaev \cite{RT} can be used
to set up a 3d TFT which has the properties expected
for Chern-Simons theory.
\item 
The vector spaces $\mathcal H_\Sigma$ can be described rather explicitly. 
Since every compact, closed, oriented 2-manifold is given by a disjoint union 
of surfaces $\Sigma_g$ of genus $g$,  it suffices to compute the dimension of 
$\mathcal H_{\Sigma_g}$. This can be done using the well-known description of
moduli spaces of flat $G$-bundles in terms of homomorphisms
from the fundamental group $\pi_1(\Sigma_g)$ to the
group $G$, modulo conjugation, 
$$ Iso \mathcal A_G(\Sigma_g) \cong \Hom( \pi_1(\Sigma_g) , G) /G $$
which can be combined with the usual description of the fundamental group
$\pi_1(\Sigma_g)$ in terms of generators and relations. In this way, one finds
that the space is one-dimensional
for surfaces of genus 0. In the case of  surfaces of genus $1$,
it is generated by pairs of commuting group elements, modulo simultaneous conjugation.
\item 
Following the same line of argument, one can show that for a closed 
3-manifold $M$, one has
$$ \big|\cala_G(M)\big| = \big|\Hom(\pi_1(M),G)\big| ~/~ |G|  \,\, . $$
This expresses the 3-manifold invariants in terms of the 
fundamental group of $M$.
\end{enumerate}
\end{Remark}

\subsection{Dijkgraaf-Witten theory as an extended TFT}

Up to this point, we have considered a version  
of Dijkgraaf-Witten theory which assigns invariants
to closed 3-manifolds $Z(M)$ and vector spaces to 
2-dimensional manifolds 
$\Sigma$. Iterating the argument that has lead us to consider
three-manifolds with boundaries, we might wish to cut the 
two-manifolds into smaller pieces as well, and thereby introduce
two-manifolds with boundaries into the picture.

Hence, we drop the requirement on the 2-manifold $\Sigma$ 
to be closed and allow $\Sigma$ to be 
a compact, oriented 2-manifold with 1-dimensional boundary 
$\partial \Sigma$. Given a field configuration 
$\phi_1 \in \mathcal A_G(\partial \Sigma)$ on the boundary 
of the surface $\Sigma$, we consider the space of all field
configurations 
$\mathcal A_G(\Sigma,\phi_1)$ on $\Sigma$ that restrict 
to the given field configuration $\phi_1$ on 
the boundary $\partial \Sigma$. Again, we linearize the
situation and consider for each field configuration
$\phi_1$ on the 1-dimensional boundary $\partial \Sigma$ 
the vector space freely generated by the isomorphism classes of field
configurations on $\Sigma$, 
$$ \mathcal H_{\Sigma, \phi_1} := " L^2\big(\mathcal A_G(\Sigma,\phi_1)\big)
" = \mathbb{C} \big<\delta_\phi \mid \phi \in Iso \mathcal A_G(\Sigma,\phi_1)\big>.$$

The object we associate to the 1-dimensional boundary
$\partial\Sigma$ of a 2-manifold $\Sigma$ is thus a map 
$\phi_1 \mapsto \mathcal H_{\Sigma, \phi_1}$ of
field configurations to vector spaces, i.e.\
a complex vector bundle over the space of all fields on the
boundary. In the case of a
finite group $G$, we prefer to see these vector bundles 
as objects of the functor category from the essentially
small category $\mathcal A_G(\partial \Sigma)$ to
the category $\vect_\mathbb{C}$ of finite-dimensional
complex vector spaces, i.e.\ as an element of
$$ \vect(\mathcal A_G(\partial \Sigma)) = \big[ \mathcal A_G(\partial \Sigma) , \vect_\mathbb{C} \big]. $$

Thus the extended version of the theory assigns the category 
$Z(S) =[\mathcal A_G(S),\vect_\mathbb{C}]$ to a
one dimensional, compact oriented manifold $S$.  These categories possess 
certain additional properties which can be summarized by saying that they 
are 2-vector spaces in the sense of {\cite{KV94}}:

\begin{Definition}\label{twovect}
\begin{enumerate}
\item A \textit{2-vector space} (over a field $\KK$) is a $\KK$-linear, abelian, 
finitely semi-simple category. Here finitely semi-simple means that
the category has finitely many isomorphism classes of simple objects and 
each object is a finite direct sum of simple objects.
\item
Morphisms between 2-vector spaces are $\KK$-linear functors and 2-morphisms 
are natural transformations. We denote the 2-category of 2-vector spaces 
by $2\vect_\KK$
\item
The \textit{Deligne tensor product} $~ \boxtimes ~$ endows $2\vect_\KK$ 
with the structure of a symmetric monoidal 2-category.
\end{enumerate}
\end{Definition}

For the Deligne tensor product, we refer to \cite[Sec. 5]{de90} or 
\cite[Def. 1.1.15]{BKLec}. The definition and the properties of symmetric 
monoidal bicategories (resp.\ 2-categories) can be found in
\cite[ch. 3]{priesPhD}.

In the spirit of definition \ref{def:2.2}, we formalize the properties of 
the extended theory $Z$ 
by describing it as a functor from a cobordism 2-category to the algebraic category 
$2\vect_\KK$. It remains to state the formal definition of the relevant geometric 
category.  Here, we ought to be a little bit more careful, since we expect 
a 2-category and hence can not identify diffeomorphic 2-manifolds. 
For precise statements on how to address the difficulties in gluing smooth 
manifolds with corners, we refer to \cite[4.3]{Morton06}; here, we
restrict ourselves to the following short definition:

\begin{Definition}\label{cob123}
$\cob(1,2,3)$ is the following symmetric monoidal bicategory:
\begin{itemize}\setlength{\itemsep}{-0.5ex}
\item { Objects} are compact, closed, oriented 1-manifolds $S$.
\item {1-Morphisms} are 
2-dimensional, compact, oriented collared cobordisms
$S \times I \hookrightarrow \Sigma \hookleftarrow  
S' \times I$.
\item {2-Morphisms} are generated by diffeomorphisms of cobordisms 
fixing the collar and 3-dimensional collared, oriented cobordisms with 
corners $M$, up to diffeomorphisms preserving the orientation and boundary.
\item Composition is by gluing along collars. 
\item The monoidal structure is given by disjoint union
with the empty set $\emptyset$ as the monoidal unit.
\end{itemize}
\end{Definition}

\begin{Remark}\label{smoothgluing}
The 1-morphisms are defined as collared surfaces, since
in the case of extended cobordism categories, we consider
surfaces rather than diffeomorphism classes of surfaces.
A choice of collar is always possible, but not unique.
The choice of collars ensures that the glued surface has 
a well-defined smooth structure. Different choices for the
collars yield equivalent 1-morphisms in $\cob(1,2,3)$.
\end{Remark}

Obviously, extended cobordism categories can be defined in 
dimensions different from three as well. We are now ready to give
the definition of an extended TFT which goes essentially 
back to Lawrence 
\cite{Lawrence}:

\begin{Definition}\label{ext}
An \textit{extended 3d TFT} is a weak symmetric monoidal 2-functor
$$ Z: \quad \cob(1,2,3) \to 2\vect_{\KK} \,\,\, . $$
\end{Definition}

We pause to explain 
in which sense extended TFTs extend the TFTs defined in
definition \ref{def:2.2}. To this end, we note that the monoidal
2-functor $Z$ has to send the monoidal unit in 
$\cob(1,2,3)$ to the monoidal 
unit in $2\vect_{\KK}$. The monoidal unit in $\cob(1,2,3)$ 
is the empty set 
$\emptyset$, and the unit in $2\vect_\KK$ is the category 
$\vect_\KK$. The functor $Z$ 
restricts to a functor $Z|_\emptyset$ from the 
endomorphisms of $\emptyset$ 
in $\cob(1,2,3)$ to the endomorphisms of $\vect_\KK$ in $2\vect_\KK$. 
It follows directly from the definition that 
$\End_{\cob(1,2,3)}\big(\emptyset\big) \cong \cob(2,3)$. 
Using the fact that 
the morphisms in $2\vect_\KK$ are additive (which follows from 
$\mathbb{C}$-linearity of functors in the definition 
of 2-vector spaces),
it is also easy to see that the equivalence of categories
$\End_{2\vect_\KK}\big(\vect_\KK\big) \cong \vect_\KK$
holds. Hence we have deduced:

\begin{Lemma}\label{restriction}
Let $Z$ be an extended 3d TFT. Then $Z|_{\emptyset}$ is a 
3d TFT in the sense of definition \ref{def:2.2}.
\end{Lemma}

At this point, the question arises whether a given 
(non-extended) 3d TFT can be 
extended. In general, there is no reason for this to be true.
For Dijkgraaf-Witten theory, however, such an extension can be
constructed based on ideas which we described at the beginning of this section.
A very conceptual presentation of this
this construction based on important ideas of \cite{Freed92}
and \cite{FQ93} can be found in 
\cite{Morton}. 
Before we describe this construction in more detail in 
subsection \ref{lin}, we first state the result:

\begin{Proposition}\cite{Morton}\label{exDW}
Given a finite group $G$, there exists an extended 3d TFT  
$Z_G$ which assigns the categories
$$  \big[ \mathcal A_G(S) , \vect_\KK \big] $$
to 1-dimensional, closed oriented manifolds $S$ and whose restriction 
$Z_G|_\emptyset$ is (isomorphic to) the Dijkgraaf-Witten TFT 
described in proposition \ref{DW}.
\end{Proposition}

\begin{Remark}
One can iterate the procedure of extension and introduce 
the notion of a fully extended TFT which also assigns
quantities to points rather than just 1-manifolds.
It can be shown that Dijkgraaf-Witten 
theory can be 
turned into a fully extended TFT, see \cite{FHLT}. The full
extension will not be needed in the present article.
\end{Remark}

\subsection{Construction via 2-linearization}\label{lin}

In this subsection, we describe in detail 
the construction of the extended 3d TFT of 
proposition \ref{exDW}. An impatient reader may skip 
this subsection and should still be able to understand 
most of the paper. 
He might, however, wish to take notice of the technique 
of 2-linearization  in proposition \ref{lin2} which is also an 
essential ingredient in our 
construction of equivariant Dijkgraaf-Witten theory in 
sequel of this paper. 

As emphasized in particular by Morton \cite{Morton}, the 
construction of the extended TFT is 
naturally split into two steps, which have already been 
implicitly present 
in preceding sections. The first step is to assign to manifolds and 
cobordisms the configuration spaces $\cala_G$ of $G$ 
bundles. We now restrict ourselves to the case when $G$ is a 
finite group. The following fact is standard:

\begin{itemize}
\item
The assignment $M \mapsto \cala_G(M) := \Bun_G$ is a contravariant 
2-functor from the category of manifolds to the 2-category of 
groupoids. 
Smooth maps between manifolds are mapped to the corresponding pullback 
functors on categories of bundles.
\end{itemize}

A few comments are in order: for a connected manifold $M$, the category 
$\cala_G(M)$ can be replaced by the equivalent category 
given by the action groupoid $\Hom\big(\pi_1(M),G\big)//G$ 
where $G$ acts by conjugation. In particular, the category 
$\cala_G(M)$ is essentially finite,
if $M$ is compact. It should be
appreciated that at this stage no restriction is imposed on
the dimension of the manifold $M$.

The functor $\cala_G(-)$ can be evaluated on a 2-dimensional cobordism 
$S \hookrightarrow \Sigma \hookleftarrow S'$ or a 3-dimensional cobordism 
$\Sigma \hookrightarrow M \hookleftarrow \Sigma'$. It then yields diagrams of the form
\begin{eqnarray*}
\cala_G(S)       \longleftarrow  ~\cala_G(\Sigma)  \longrightarrow  \cala_G(S') \\
\cala_G(\Sigma)  \longleftarrow  \cala_G(M)       \longrightarrow  \cala_G(\Sigma').
\end{eqnarray*}
Such diagrams are called spans. They are the morphisms of 
a symmetric monoidal bicategory $\Span$ of spans of groupoids as follows (see e.g. \cite{Daw04} or \cite{Morton06}):
\begin{itemize}\setlength{\itemsep}{-1ex}
\item Objects are (essentially finite) groupoids.
\item Morphisms are spans of essentially finite groupoids.
\item 2-Morphisms are isomorphism classes of spans of span-maps.
\item Composition is given by forming weak fiber products.
\item The monoidal structure is given by the cartesian product $\times$ of groupoids.
\end{itemize}

\begin{Proposition}[\cite{Morton}]\label{conf}
$\cala_G$ induces a symmetric monoidal 2-functor 
$$\widetilde{\cala_G}: \cob(1,2,3) \to \Span.$$
This functor assigns to a 1-dimensional manifold $S$
the groupoid $\cala_G(S)$, to a 2-dimensional
cobordism $S \hookrightarrow \Sigma \hookleftarrow S'$
the span $\cala_G(S)       \longleftarrow  ~\cala_G(\Sigma)  \longrightarrow  \cala_G(S')$ and to a 3-cobordism with corners a span of
 span-maps.
\end{Proposition}
\begin{proof}
It only remains to be shown that composition of morphisms 
and the monoidal structure is respected. The first assertion 
is shown in \cite[theorem 2]{Morton} and the second 
assertion follows immediately from the fact that bundles 
over disjoint unions are given by pairs of bundles over the 
components, i.e.\ 
$\cala_G(M \sqcup M') = \cala_G(M) \times \cala_G(M')$.
\end{proof}

The second step in the construction of extended Dijkgraaf-Witten theory is the 2-linearization 
of \cite{Morton08}. As we have explained in section
\ref{sec:motivation}, the idea is to associate to a 
groupoid $\Gamma$ its category of vector bundles 
$\vect_\KK(\Gamma)$. If $\Gamma$ is essentially finite, the category of vector bundles is
conveniently defined as the functor category $\big[\Gamma,\vect_\KK\big]$. 
If $\KK$ is algebraically closed of characteristic zero, this category is a 2-vector space,
see \cite[Lemma 4.1.1]{Morton08}.

\begin{itemize}
\item
The assignment $\Gamma \mapsto \vect_\KK\big(\Gamma\big) := \big[\Gamma,\vect_\KK\big]$ is a 
contravariant 2-functor from the bicategory of (essentially finite) groupoids to the 2-category 
of 2-vector spaces. Functors between groupoids are sent to pullback functors.
\end{itemize}
We next need to explain what 2-linearization assigns to spans of groupoids.
To this end, we use the following lemma due to 
\cite[4.2.1]{Morton08}:
\begin{Lemma}
Let $f: \Gamma \to \Gamma'$ be a functor between essentially finite groupoids. Then the pullback functor
$ f^*: \vect\big(\Gamma'\big) \to \vect\big(\Gamma\big) $
admits a 2-sided adjoint $ f_*: \vect\big(\Gamma\big) \to \vect\big(\Gamma'\big)$, called the
\textit{pushforward}.

\end{Lemma}

Two-sided adjoints are also called `ambidextrous' adjoint, 
see \cite[ch. 5]{bartlett} for a discussion.
We use this pushforward to associate to a span $$ \xymatrix{
\Gamma & \Lambda \ar_-{~p_0}[l] \ar^-{p_1}[r] & \Gamma' } $$
of (essentially finite) groupoids the \textit{`pull-push'-functor} 
$$ (p_1)_* \circ (p_0)^* : \quad \vect_\KK\big(\Gamma\big) \longrightarrow \vect_\KK\big(\Gamma'\big).$$
A similar construction \cite{Morton08} associates to spans of span-morphisms a natural transformation.
Altogether we have:

\begin{Proposition}[\cite{Morton08}]\label{lin2}
The functor $\Gamma \mapsto \vect_\KK(\Gamma)$ can be 
extended to a symmetric monoidal 2-functor on the category of spans of groupoids
$$ \widetilde{\calv_\KK}: \Span \to 2\vect_{\KK}.$$
This 2-functor is called \emph{2-linearization}.
\end{Proposition}
\begin{proof}
The proof that $\widetilde{\calv_\KK} $ is a 2-functor is in 
\cite{Morton08}. The fact that $\widetilde{\calv_\KK} $ is 
monoidal follows from the fact that  
$\vect_\KK\big(\Gamma \times \Gamma'\big) \cong 
\vect_\KK\big(\Gamma\big) \boxtimes 
\vect_\KK\big(\Gamma'\big)$ for a product 
$\Gamma \times \Gamma'$ of essentially finite groupoids.
\end{proof}

Arguments similar to the ones in \cite[prop 1.10]{Daw04} 
which are based on the universal property of the span 
category can be used to show that such an extension is
essentially unique.

We are now in a position to give the functor $Z_G$ 
described in proposition \ref{exDW} which is 
Dijkgraaf-Witten theory as an extended 3d TFT 
as the composition of functors
$$ Z_G := \widetilde{\calv_\KK} \circ \widetilde{\cala_G} : 
\quad \cob(1,2,3) \longrightarrow 2\vect_\KK.\qquad$$
It follows from propositions \ref{conf} and \ref{lin2} that $Z_G$ 
is an extended 3d TFT in the sense of definition \ref{ext}. 
For the proof of proposition \ref{exDW}, it remains to be shown 
that $Z_G|_\emptyset$ is the Dijkgraaf-Witten 3d TFT from 
proposition \ref{DW}; this follows from a 
calculation which can be found in \cite[Section 5.2]{Morton}.

\subsection{Evaluation on the circle}\label{evaluation}

The goal of this subsection is a more detailed discussion 
of extended Dijkgraaf-Witten theory $Z_G$ as
described in proposition \ref{exDW}. Our focus is 
on the object assigned to the  1-manifold $\SS$ given 
by the circle with its standard orientation. 
We start our discussion by evaluating an arbitrary 
extended 3d TFT $Z$ as in definition \ref{ext}
on certain manifolds of different dimensions: 

\begin{enumerate}
\item
To the circle $\SS$, the extended TFT assigns a $\KK$-linear, abelian finitely semisimple category $\calc_Z:=Z(\SS)$.
\item
To the two-dimensional sphere with three boundary components,
two incoming and one outgoing, also known as the \emph{pair of pants},
\begin{center}
\includegraphics{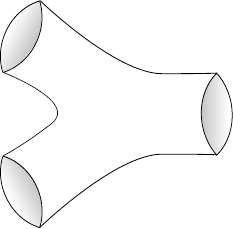}
\end{center}
the TFT associates a functor
$$\ \otimes:\,\, \calc_Z\boxtimes\calc_Z \to\calc_Z\ \,\, , $$
which turns out to provide a tensor product on the
category $\calc_Z$.
\item The figure 
\begin{center} 
\includegraphics{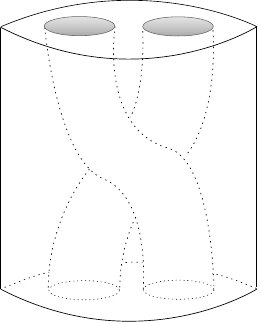}
\end{center}
shows a 2-morphism between two
three-punctured spheres, drawn as the upper and lower lid.
The outgoing circle is drawn as the boundary of the big disk.
To this cobordism, the TFT associates a natural
transformation
$$\otimes \Rightarrow \otimes^{{\rm opp}} $$
which turns out to be a braiding.
\end{enumerate}

Moreover, the TFT provides coherence cells, in particular associators and relations 
between the given structures. This endows the category $\calc_Z$ with much additional
structure. This structure can be summarized as follows:
\begin{Proposition}
For $Z$ an extended 3d TFT, the category $\calc_Z := Z(\SS)$ is naturally endowed with
the structure of a braided tensor category.
\end{Proposition}
For details, we refer to \cite{Freed92} \cite{freed94} \cite{freed99} and \cite{CY99}.
This is not yet the complete structure that can be extracted:
from the braiding-picture above it is intuitively clear that the braiding is not symmetric;
in fact, the braiding is `maximally non-symmetric' in a precise sense that is explained in 
definition \ref{modularity}. We discuss this in the next section for the category obtained 
from the Dijkgraaf-Witten extended TFT. 

\medskip

We now specialize to the case of extended Dijkgraaf-Witten 
TFT $Z_G$. We first determine
the category $\calc(G) := \calc_Z$; it is by definition 
$$ \calc(G) = \big[ \mathcal A_G(\SS) , \vect_\KK \big]. $$
It is a standard result in the theory of coverings that
$G$-covers on $\SS$ are described by group homomorphisms $\pi_1(\SS)
\to G$ and their morphisms by group elements acting by
conjugation. Thus the category $\cala_G(\SS)$
is equivalent to the action groupoid  
$G//G$ for the conjugation action. As a consequence,
we obtain the abelian category 
$\calc(G) \cong [G // G, \vect_\KK ]$. 
We spell out this functor category explicitly:

\begin{Proposition}\label{DWcat}
For the extended Dijkgraaf-Witten 3d TFT, the category 
$\calc(G)$ associated to the circle $\SS$ is given by 
the category of $G$-graded vector spaces $V = \bigoplus_{g \in G} V_g$ 
together with a $G$-action on $V$ such that for all
$x,y\in G$
$$x . V_g \subset V_{xgx^{-1}}\,\,\, .$$
\end{Proposition}

As a next step we determine the tensor product on $\calc(G)$. 
Since the fundamental group of the pair 
of pants is the free group on two generators, the relevant 
category of $G$-bundles is 
equivalent to the action groupoid $(G \times G) // G$ where $G$ 
acts by simultaneous conjugation on the two copies of $G$. 
The 2-linearization $\widetilde{\calv_\KK}$ on the span 
$$(G//G) \times (G//G) \leftarrow (G\times G) //G \rightarrow G//G.$$
is treated in detail in \cite[rem.\ 5]{Morton}; the result 
of this calculation yields the following tensor product:
\begin{Proposition}\label{prop:2.17}
The tensor product of $V$ and $W$ is given by the $G$-graded vector space
$$ (V \otimes W)_g = \bigoplus_{st =g} V_s \otimes W_t $$
together with the $G$-action $g. (v,w) = (gv, gw)$. The associators are the 
obvious ones induced by the tensor product in $\vect_\KK$.
\end{Proposition}
In the same vein, the braiding can be calculated: 
\begin{Proposition} \label{prop:2.18}
The braiding $V \otimes W \to W \otimes V$ is for $v \in V_g$ and $w \in W$ given by
$$ v \otimes w \mapsto gw \otimes v.$$
\end{Proposition}

\subsection{Drinfel'd double and modularity}\label{dd}

The braided tensor category $\calc(G)$ we just computed
from the last section has a well-known description as the 
category of modules over
a braided Hopf-algebra $\cald(G)$, 
the Drinfel'd double $\cald(G):=\cald(\KK[G])$ of the group algebra
$\KK[G]$ of $G$, see e.g.\ \cite[Chapter 9.4]{kassel1995quantum}
The Hopf-algebra $\cald(G)$ is defined as follows:

As a vector space, $\cald(G)$ is the tensor product $\KK(G) \otimes \KK[G]$ of the algebra of functions on $G$
and the group algebra of $G$, i.e.\ we have the canonical basis $(\delta_g \otimes h)_{g,h\in G}$. The algebra structure can be described as a 
smash product (\cite{montgomery1993hopf}), 
an analogue of the semi-direct product for groups: in the canonical basis, we have
\[(\delta_{g} \otimes h)(\delta_{g'}\otimes h') 
= \left\{\begin{array}{l} \delta_{g}\otimes hh'
\text{ for } g=hg'h^{-1} \\[.2em]
0 \text{ else.}
\end{array}\right.
\]
where the unit is given by the tensor product of the two units:
$\sum_{g \in G} \delta_g \otimes 1$.
The coalgebra structure of $\cald(G)$ is given by the 
tensor product of the coalgebras $\KK(G)$ 
and $\KK[G]$, i.e.\ the coproduct reads
\[\Delta(\delta_g \otimes h) = \sum_{g'g'' = g} (\delta_{g'}\otimes h) \otimes (\delta_{g''}\otimes h)\]
and the counit is given by
$\epsilon(\delta_1 \otimes h) = 1$ and
$\epsilon(\delta_g \otimes h) = 0$ for $g\neq 1$ for all
$h\in G$.
It can easily be checked that this defines a bialgebra structure on $\KK(G) \otimes \KK[G]$ and that furthermore the linear map
\[S:(\delta_g \otimes h) \mapsto (\delta_{h^{-1}g^{-1}h} \otimes h^{-1})\]
is an antipode for this bialgebra so that $\cald(G)$ is
a Hopf algebra. Furthermore, the element
\[R := \sum_{g,h\in G} (\delta_g \otimes 1)\otimes (\delta_h \otimes g) \in \cald(G)\otimes \cald(G)\]
is a universal R-matrix, which fulfills the defining identities of a braided bialgebra and corresponds to the braiding in
proposition \ref{prop:2.18}. At last, the element
\[\theta := \sum_{g \in G} (\delta_g \otimes  g^{-1}) \in \cald(G)\]
is a ribbon-element in $\cald(G)$, which gives $\cald(G)$ 
the structure of a 
ribbon Hopf-algebra (as defined in \cite[Definition 14.6.1]{kassel1995quantum}).
Comparison with propositions \ref{prop:2.17} and
\ref{prop:2.18} shows

\begin{Proposition}
The category $\calc(G)$ is isomorphic, as a braided tensor category, to the category $\cald(G)\mod$.
\end{Proposition}

The category $\cald(G)\mod$ is actually endowed with more structure than the one of a braided monoidal 
category. Since $\cald(G)$ is a ribbon Hopf-algebra, the category of representations $\cald(G)\mod$ has 
also dualities and a compatible twist, i.e. has the structure of a ribbon 
category (see \cite[Proposition 16.6.2]{kassel1995quantum} 
or \cite[Def. 2.2.1]{BKLec} for the notion of a ribbon category). Moreover, the category $\cald(G)\mod$ is a 
2-vector space over $\KK$ and thus, in particular, finitely semi-simple. We finally make explicit the 
non-degeneracy condition on the braiding that was mentioned in the last subsection.

\begin{Definition}\label{modularity}
\begin{enumerate}
\item
Let $\KK$ be an algebraically closed field of characteristic zero.
A \emph{premodular tensor category} over $\KK$ is a
$\KK$-linear, abelian, finitely semisimple category
$\calc$ which has the structure of a ribbon category 
such that the tensor product is linear in each variable and the tensor unit is absolutely
simple, i.e.\ $\mathrm{End}(\mathbf{1}) = \KK$.
\item
Denote by $\Lambda_{\calc}$ a set of representatives for the
isomorphism classes of simple objects. The braiding on $\calc$ 
allows to define the 
\emph{S-matrix} with entries in the field $\KK$
\[s_{XY} := tr(R_{YX} \circ R_{XY}) \,\, ,\]
where $X,Y \in \Lambda_{\calc}$.
A premodular category is called \emph{modular}, if the S-matrix is invertible.
\end{enumerate}
\end{Definition}

In the case of the Drinfel'd double, the S-matrix can be
expressed explicitly in terms of characters of finite
groups \cite[Section 3.2]{BKLec}. Using orthogonality relations, one shows:

\begin{Proposition}\label{modular-Drinfeld}
The category $\calc(G) \cong \cald(G)\mod$ is modular.
\end{Proposition}

The notion of a modular tensor category first arose as a formalization of
the Moore-Seiberg data of a two-dimensional rational conformal field theory.
They are the input for the Turaev-Reshetikhin construction
of three-dimensional topological field theories.

\section{Equivariant Dijkgraaf-Witten theory}

We are now ready to turn to the construction of equivariant
generalization of the results of section \ref{sec:DW}. 
We denote again by $G$ a finite group.
Equivariance will be 
with respect to another finite group $J$ that acts on $G$ in a way
we will have to explain. As usual, `twisted sectors'
\cite{VW95} have to be taken into account for a consistent equivariant
theory. A description of these twisted sectors in terms
of bundles twisted by $J$-covers is one important
result of this section.

\subsection{Weak actions and extensions}

Our first task is to identify the appropriate definition of a $J$-action. 
The first idea that comes to mind -- a genuine action of 
the group $J$ acting on $G$ by group automorphisms --
turns out to need a modification. For reasons that will become apparent 
in a moment, we only require an action up to inner automorphism.

\begin{Definition}\label{Def:action}
\begin{enumerate}
\item
A \emph{weak action} of a group $J$ on a group $G$ consists 
of a collection of group automorphisms $\rho_j: G \to G$, 
one automorphism for each $j \in J$, and a collection of group
elements $c_{i,j} \in G$, one group element for each pair 
of elements $i,j \in J$. These data are required to obey
the relations:
$$
  \rho_i \circ \rho_j = \Inn_{c_{i,j}} \circ \rho_{ij}  \qquad
  \rho_i(c_{j,k}) \cdot c_{i,jk}  = c_{i,j} \cdot c_{ij,k} \quad \text{ and } \quad c_{1,1} = 1
$$
for all $i,j,k \in J$. Here $\Inn_g$ denotes the inner 
automorphism $G \to G$ associated to an element $g \in G$. 
We will also use the short hand notation $ ~^j\!g := \rho_j(g)$.
\item Two weak actions $\big(\rho_j,c_{i,j})$ 
and $\big(\rho'_j,c'_{i,j})$ of a group $J$ on a group $G$ 
are called isomorphic, if there is a collection of group elements 
$h_j \in G$, one group element for each $j \in J$, such that
$$ \rho_j' = \Inn_{h_j} \circ \rho_j \quad \text{ and } c'_{ij} \cdot h_{ij} = h_i \cdot \rho_i(h_j) \cdot c_{ij}$$
\end{enumerate}
\end{Definition}

\begin{Remark}\label{rem:3.2}
\begin{enumerate}
\item 
If all group elements $c_{i,j}$ equal the neutral element,
$c_{i,j} = 1$, the weak action reduces to a strict action 
of $J$ on $G$ by group automorphisms. 

\item A weak action induces a strict action of $J$ on the 
group  $\Out(G)=\Aut(G)/\Inn(G)$ of outer automorphisms.

\item
In more abstract terms, 
a weak action amounts to a (weak) 2-group homomorphism 
$J \to \AUT(G)$. Here
$\AUT(G)$ denotes the automorphism 
2-group of $G$. This automorphism 2-group can be described 
as the monoidal category of endofunctors of the 
one-object-category with morphisms $G$. The group $J$
is considered as a discrete 2-group with only identities 
as morphisms. For more details 
on 2-groups, we refer to \cite{baez04}.
\end{enumerate}
\end{Remark}

Weak actions are also known under the name \textit{Dedecker cocycles},
due to the work \cite{Ded60}. 
The correspondence between weak actions and extensions 
of groups is also termed \textit{Schreier theory},
with reference to \cite{Schreier}. Let us briefly sketch 
this correspondence:
\begin{itemize}
\item
Let $\big(\rho_j, c_{i,j}\big)$ be a weak action of $J$ on $G$. 
On the set $H := G \times J$, we define a multiplication by
\begin{equation} \label{mult}
(g,i) \cdot (g',j) := \big(g \cdot ~^i (g') \cdot c_{i,j}~, ~ij\big). 
\end{equation}
One can check that this turns $H$ into a group in such 
a way that the sequence $G \to H \to J$ consisting of the 
inclusion $g \mapsto (g,1)$ and the projection 
$(g,j) \mapsto j$ is exact.

\item 
Conversely, let $G \longrightarrow H \stackrel{\pi}{\longrightarrow} J$ 
be an extension of groups. Choose a set theoretic section 
$s: J \to H$ of $\pi$ with $s(1) = 1$.  Conjugation with
the group element $s(j)\in H$ leaves the normal
subgroup $G$ invariant. We thus obtain for  
$j\in J$\ the automorphism
$\rho_j(g) := s(j)~ g ~s(j)^{-1}$ of $G$. 
Furthermore, the element $c_{i,j} := s(i) s(j) s(ij)^{-1}$ 
is in the kernel of $\pi$ and thus actually contained in 
the normal subgroup $G$. It is then straightforward to 
check that $\big(\rho_j,c_{i,j}\big)$ defines a weak 
action of $J$ on $G$.

\item Two different set-theoretic sections $s$ and $s'$ of the extension 
$G \to H \to J$ differ by a map $J \to G$. This map 
defines an isomorphism of the induced weak actions
in the sense of definition \ref{Def:action}.2.
\end{itemize}

We have thus arrived at the

\begin{Proposition}[Dedecker, Schreier]
\label{dedeschreier}
There is a 1-1 correspondence between isomorphism classes of weak actions 
of $J$ on $G$ and isomorphism classes of group extensions $G \to H \to J$.
\end{Proposition}

\begin{Remark}
\begin{enumerate}
\item 
One can easily turn this statement into an 
equivalence of categories. Since we do not
need such a statement in this paper,
we leave a precise formulation to the reader.

\item
Under this correspondence, 
\emph{strict actions} of $J$ 
on $G$ correspond to \emph{split extensions}. This can be easily seen as follows: given a \emph{split} extension $G \to H \to J$, one can choose
the section $J\to H$ as a group homomorphism and thus
obtains a \emph{strict} action of $J$ on $G$. Conversely for a strict action of $J$ on $G$ it is easy to see that the group
constructed in equation (\ref{mult}) is a semidirect
product  and thus the sequence of groups splits. 
To cover all extensions, we thus really
need to consider weak actions.
\end{enumerate}
\end{Remark}

\subsection{Twisted bundles}\label{sec:twisted}

It is a common lesson from field theory that in an equivariant
situation, one has to include ``twisted sectors'' to
obtain a complete theory.
Our next task is to construct the parameters labeling 
twisted sectors for a given weak action of a 
finite group $J$ on $G$, with corresponding extension 
$G \to H \to J$ of groups and chosen set-theoretic section
$J\to H$. We will adhere to a two-step procedure as
outlined after proposition \ref{lin2}. To this end,
we will first construct for any smooth manifold a
category of twisted bundles. Then, the linearization functor
can be applied to spans of such categories. 

We start our discussion of twisted $G$-bundles with 
the most familiar case of the circle, $M=\SS$.

The isomorphism classes of $G$-bundles on $\SS$ are in bijection to connected 
components of the free loop space $\call BG$ of the classifying space $BG$:
$$ Iso\big(\cala_G(\SS)\big) = \Hom_{\mathrm{Ho(Top)}}(\SS,BG) = \pi_0(\call BG). $$
Given a (weak) action of $J$ on $G$, one can introduce
twisted loop spaces. For any element $j\in J$, we have a group automorphism $j: G \to G$ 
and thus a homeomorphism $j: BG \to BG$. The $j$-twisted loop space is then defined to be
$$ \call^j BG :=\big\{ f:[0,1]\to BG \mid f(0)=j \cdot f(1) \big\}. $$
Our goal is to introduce for every group element $j\in J$ 
a category $\cala_G(\SS,j)$ of $j$-twisted $G$-bundles on $\SS$ such that 
$$ Iso\big(\cala_G(\SS,j)\big) = \pi_0(\call^j BG) \,\,\, . $$

In the case of the circle $\SS$, the twist parameter was a group element
$j\in J$. A more geometric description uses a family of $J$-covers $P_j$ 
over $\SS$, with $j \in J$. The cover $P_j$ is uniquely determined by 
its monodromy $j$ 
for the base point $1 \in \SS$ and a fixed point in the fiber over $1$.
A concrete construction of the cover $P_j$ is given by the quotient
$P_j := [0,1] \times J / \sim$ where $(0,i) \sim (1,ji)$ for all $i \in J$. 
In terms of these $J$-covers, we can write
$$ \call^j BG = \big\{ f: P_j \to BG \mid f \text{ is $J$-equivariant} \big\}.$$

This description generalizes to an arbitrary smooth manifold $M$.
The natural twist parameter in the general case is 
a $J$-cover $P\stackrel J\to M$. 

Suppose, we have a weak $J$-action on $G$ and construct
the corresponding extension $G\to H \stackrel\pi\to J$.
The category of bundles we 
need are $H$-lifts of the given $J$-cover:

\begin{Definition}\label{def:twisted}
Let $J$ act weakly on $G$.
Let $P\stackrel J\to M$ be a $J$-cover over $M$.
\begin{itemize}
\item
A $P$-twisted $G$-bundle over $M$ is a pair $(Q,\varphi)$, consisting of an 
$H$-bundle $Q$ over $M$ 
and a smooth map $\varphi: Q \to P$ over $M$ that is required to obey
$$ \varphi(q\cdot h) = \varphi(q) \cdot \pi(h) \,\,\,  $$
for all $q\in Q$ and $h\in H$. Put differently, a
$P\stackrel J\to M$-twisted $G$-bundle is
a lift of the $J$-cover $P$ 
reduction along the group homomorphism
$\pi:\ H\to J$.

\item 
A morphism of $P$-twisted bundles $(Q,\varphi)$ and $(Q',\varphi')$ is a morphism 
$f: Q \to Q'$ of $H$-bundles such that $\varphi' \circ f = \varphi$.

\item We denote the category of $P$-twisted $G$-bundles by $\cala_G\big(P \to M\big)$.
For $M= \SS$, we introduce the abbreviation $\cala_G\big(\SS,j) := \cala_G\big(P_j \to \SS\big)$ for the standard covers of the circle.
\end{itemize}
\end{Definition}

\begin{Remark}
There is an alternative point of view on a $P$-twisted bundle $(Q,\varphi)$:
the subgroup $G\subset H$ acts on the total space $Q$ in such a way that the map $\varphi: Q \to P$ 
endows $Q$ with the structure of a $G$-bundle on $P$. Both the structure group
$H$ of the bundle $Q$ and the bundle $P$ itself carry an action of $G$; for twisted
bundles, an equivariance condition on this action has to be imposed.
Unfortunately this equivariance property is relatively involved; therefore, we have opted
for the definition in the form given above.
\end{Remark}

A morphism $f: P \to P'$ of $J$-covers over the same manifold induces a functor 
$f_*: \cala_G\big(P \to M\big) \to \cala_G\big(P' \to M\big)$ by $f_*(Q,\varphi) := (Q,f \circ \varphi)$. 
Furthermore, for a smooth map $f: M \to N$, we can pull back the twist data $P\to M$ and
get a \textit{pullback functor} of twisted $G$-bundles:
$$f^*: \cala_G\big(P \to N\big) \to \cala_G\big(f^*P \to M\big)$$
by $f^*(Q,\varphi)= (f^*Q,f^*\varphi)$. 
Before we discuss more sophisticated properties of twisted 
bundles, we have to
make sure that our definition is consistent with `untwisted' bundles:

\begin{Lemma} \label{trivbundle}
Let the group $J$ act weakly on the group $G$. For $G$-bundles 
twisted by the trivial $J$-cover $M\!\times\!J \to M$, we 
have a canonical equivalence of categories
$$
\cala_G\big(M\!\times\!J \to M\big)\cong \cala_G(M) .
$$
\end{Lemma}

\begin{proof}
We have to show that for an element $(Q,\varphi) \in 
\cala_G\big(M\!\times\!J \to M\big)$ the $H$-bundle $Q$ 
can be reduced to a $G$-bundle. Such a reduction is the 
same as a section of the associated fiber bundle 
$\pi_*(Q) \in \Bun_J(M)$
see e.g. \cite[Satz 2.14]{baum}). 
Now $\varphi: Q\to M\times J$ induces an isomorphism
of $J$-covers $Q\times_H\ J\cong (M\times J)\times_H\ J
\cong M\times J$ so that the bundle $Q\times_H\ J$ 
is trivial as a $J$-cover and in particular admits
global sections.

Since morphisms of twisted bundles have to commute with these 
sections, we obtain in that way a functor 
$\cala_G\big(M\!\times\!J \to M\big) \to \cala_G(M)$. Its inverse 
is given by extension of $G$-bundles on $M$ to $H$-bundles on $M$. 
\end{proof}

We also give a description of twisted bundles using 
standard covering theory; for an alternative description 
using  \v{C}ech-cohomology, we refer to appendix \ref{cech}. 
We start by recalling the following standard fact
from covering theory, see e.g.\ \cite[1.3]{hatcher}
that has already been used to prove proposition \ref{DWcat}:
for a finite group $J$, the category of $J$-covers
is equivalent to the action groupoid
$\Hom(\pi_1(M),J) //J$. (Note that this equivalence
involves choices and is not canonical.)

To give a similar description of twisted bundles, fix
a $J$-cover $P$. Next, we choose
a basepoint $m \in M$ and a point $p$ in the fiber $P_m$
over $m$. These data determine a unique group
morphism $\omega: \pi_1(M,m) \to J$ representing $P$.

\begin{Proposition}\label{covering}
Let $J$ act weakly on $G$.
Let $M$ be a connected manifold and $P$ be a $J$-cover
over $M$ represented after the choices just indicated
by the group homomorphism $\omega: \pi_1(M) \to J$. 
Then there is a  (non-canonical) equivalence of categories
$$ \cala_G\big(P \to M \big) \cong \Hom^\omega\big(\pi_1(M),H\big) // G $$
where we consider group homomorphisms
$$ \Hom^\omega\big(\pi_1(M),H\big) := \big\{\mu: \pi_1(M) \to H \mid \pi \circ \mu = \omega \big\} $$
whose composition restricts to the group
homomorphism $\omega$ describing the $J$-cover $P$. The group
$G$ acts on $\Hom^\omega\big(\pi_1(M),H\big)$
via pointwise conjugation using the inclusion 
$G \to H$.
\end{Proposition}
\begin{proof}
Let $m\in M$ and $p\in P$ over $m$ be the choices of
base point in the $J$-cover $P\to M$ that lead to
the homomorphism $\omega$. Consider a $(P\to M)$ twisted
bundle $Q\to M$. Since $\varphi:\ Q\to P$ is
surjective, we can choose a base point $q$ in the fiber
of $Q$ over $m$ such that $\varphi(q)=p$. The
group homomorphism $\pi_1(M)\to H$ describing the
$H$-bundle $Q$
is obtained by lifting closed paths in $M$ starting in
$m$ to paths in $Q$ starting
in $q$. They are mapped under $\varphi$ to
lifts of the same path to $P$ starting at $p$, and
these lifts are just described by the group homomorphism
$\omega: \pi_1(M)\to J$ describing the cover $P$.
If the end point of the path in $Q$ is
$qh$ for some $h\in H$, then by the defining property
of $\varphi$, the lifted path in $P$ has endpoint
$\varphi(qh)=\varphi(q)\pi(h)= p\pi(h)$. Thus
$\pi\circ\mu=\omega$.

\end{proof}

\begin{Remark}
For non-connected manifolds, a description as in proposition
\ref{covering} can be obtained for every component.
Again the equivalence involves choices of base points on
$M$ and in the fibers over the base points. This could
be fixed by working with pointed manifolds, but pointed
manifolds cause problems when we consider cobordisms.
Alternatively, we could use the fundamental groupoid
instead of the fundamental group, see e.g. \cite{may}. 
\end{Remark}

\begin{Example}\label{ex_twisted}
We now calculate the categories of twisted bundles over 
certain manifolds using proposition \ref{covering}. 
\begin{enumerate}\setlength{\itemsep}{-0.5ex}
\item For the circle $\SS$,
$\omega\in\Hom(\pi_1(\SS),J)=\Hom(\Z,J)$ is determined
by an element $j\in J$ and the condition $\pi\circ\mu=
\omega$ requires $\mu(1)\in H$ to be in the preimage
$H_j:=\pi^{-1}(j)$ of $j$. Thus,  we have $\cala_G(\SS,j)
\cong H_j // G$.
\item For the 3-Sphere $\mathbb{S}^3$, all twists $P$ and 
all $G$-bundles are trivial. Thus, we have 
$\cala_G(P \to \mathbb{S}^3) \cong \cala_G(\mathbb{S}^3) 
\cong pt // G$. 
\end{enumerate}
\end{Example}

\subsection{Equivariant Dijkgraaf-Witten theory}\label{subsec:eqDW}

The key idea in the construction of equivariant
Dijkgraaf-Witten theory is to take twisted bundles 
$\cala_G(P \to M)$ as the field configurations, taking
the place of $G$-bundles  in section \ref{sec:DW}. 
We cannot expect to get then invariants of closed
3-manifolds $M$, but rather invariants of 3-manifolds $M$ 
together with a twist datum, i.e.\ a $J$-cover
$P$ over $M$. Analogous statements apply to
manifolds with boundary and cobordisms. Therefore we 
need to introduce extended cobordism-categories as 
$\cob(1,2,3)$ in definition \ref{cob123}, but endowed
with the extra datum of a $J$-cover over each manifold.

\begin{Definition}\label{defcobj}
$\cob^J(1,2,3)$ is the following symmetric monoidal bicategory:
\begin{itemize}\setlength{\itemsep}{-0.5ex}
\item { Objects} are compact, closed, oriented 1-manifolds $S$, together with a $J$-cover $P_S \stackrel{J}{\to} S$.
\item {1-Morphisms} are collared cobordisms 
$$S \times I \hookrightarrow \Sigma \hookleftarrow  S' \times I$$
where $\Sigma$ is a 2-dimensional, compact, oriented cobordism, together 
with a $J$-cover $P_\Sigma \to \Sigma$ and isomorphisms 
$$ P_\Sigma|_{(S \times I)} \iso P_S \times I \qquad \text{ and } \qquad P_\Sigma|_{(S' \times I)} \iso P_{S'} \times I.$$
over the collars.
\item {2-Morphisms} are generated by 
\begin{enumerate}[-]\setlength{\itemsep}{-0.5ex}
\item orientation preserving diffeomorphisms $\varphi: \Sigma \to \Sigma'$ of cobordisms fixing the collar together with an 
isomorphism $\widetilde\varphi: P_\Sigma \to P_{\Sigma'}$ covering $\varphi$.
\item 3-dimensional collared, oriented cobordisms with corners $M$ with 
cover $P_M \to M$ together with covering isomorphisms over the collars 
(as before) up 
to diffeomorphisms preserving the orientation and boundary. 
\end{enumerate}
\item Composition is by gluing cobordisms and covers along collars. 
\item The monoidal structure is given by disjoint union.
\end{itemize}
\end{Definition}

\begin{Remark}
In analogy to remark \ref{smoothgluing}, we point out
that the isomorphisms of covers are defined 
over the collars, rather than only over the the boundaries. 
This endows the glued cover with a well-defined smooth structure. 
\end{Remark}

\begin{Definition}\label{defjtft}
An extended 3d $J$-TFT is a symmetric monoidal 2-functor
$$ Z: \cob^J(1,2,3) \to 2\vect_\KK.$$
\end{Definition}

Just for the sake of completeness, we will also give a 
definition of non-extended $J$-TFT. Therefore define the 
symmetric monoidal category 
$\cob^J(2,3)$ to be the endomorphism category of the 
monoidal unit $\emptyset$ in $\cob(1,2,3)$. More concretely, this 
category has as 
objects closed, oriented 2-manifolds with $J$-cover and as 
morphisms $J$-cobordisms between them.
\begin{Definition}
A (non-extended) 3d $J$-TFT is a symmetric monoidal 2-functor
$$ \cob^J(2,3) \to \vect_\KK.$$
\end{Definition}
Similarly as in the non-equivariant case (lemma \ref{restriction}), 
we get 
\begin{Lemma}
Let $Z$ be an extended 3d $J$-TFT. Then $Z|_{\emptyset}$ is a 
(non-extended) 3d $J$-TFT.
\end{Lemma}

Now we can state the main result of this section:
\begin{Theorem} \label{Thm1}
For a finite group $G$ and a weak $J$-action on $G$, there is 
an extended 3d $J$-TFT called $Z_G^J$ 
which assigns the categories
$$ \vect_\KK\big(\mathcal A_G(P \to S)\big) = 
\big[ \mathcal A_G(P \to S) , \vect_\KK \big] $$
to 1-dimensional, closed oriented manifolds $S$ with $J$-cover $P \to S$.
\end{Theorem}
We will give a proof of this theorem in the next sections.
Having twisted bundles at our disposal, 
the main ingredient will again be the 2-linearization 
described in section \ref{lin}.
\subsection{Construction via spans}\label{subsec:spans}

As in the case of ordinary Dijkgraaf-Witten theory, cf.\ section 
\ref{lin}, equivariant Dijkgraaf-Witten $Z_G^J$
theory is constructed as the composition of the symmetric monoidal 2-functors
$$\widetilde{\cala_G}: \cob^J(1,2,3) \to \Span \quad \text{ and } 
\quad \widetilde{\calv_\KK} : \Span \to 2\vect_\KK.$$
The second functor will be exactly the 2-linearization 
functor of proposition \ref{lin2}. 
Hence we can limit our discussion to the construction of the first 
functor $\widetilde{\cala_G}$. 
As it will turn out, our definition of twisted bundles is 
set up precisely in such a way that the
construction of the corresponding functor in proposition \ref{conf} can be generalized. \\

Our starting point is the following observation:
\begin{itemize}
\item
The assignment $(P_M \stackrel{J}{\to} M) \longmapsto 
\cala_G(P_M \stackrel J\to M)$ of 
twisted bundles to a twist datum $P_M\to M$ constitutes 
a contravariant 2-functor from the  category of manifolds with 
$J$-cover to the 2-category of groupoids. Maps between manifolds with cover are mapped to 
the corresponding pullback functors of bundles.
\end{itemize}

From this functor which is defined on manifolds of any dimension, we  
construct a functor $\widetilde{\cala_G}$ on $J$-cobordisms
with values in the 2-category $\Span$ of spans of groupoids,
where the category $\Span$ is defined in section \ref{lin}.  
To an object in $\cob^J(1,2,3)$, i.e.\
to a $J$-cover $P_S\to M$, we assign 
the category $\cala_G(P_S \to S)$ of $J$-covers. To a 1-morphism
$ P_S \hookrightarrow P_\Sigma \hookleftarrow  P_S'$ in 
$\cob^J(1,2,3)$, we associate the span
\begin{equation}\label{span1}
  \cala_G(P_S \to S) \leftarrow \cala_G(P_\Sigma \to \Sigma) \rightarrow \cala_G(P_{S'} \to S')
\end{equation}
and to a 2-morphism of the type $P_\Sigma \hookrightarrow P_M \hookleftarrow P_{\Sigma'}$ the span
\begin{equation}\label{span2}
  \cala_G(P_\Sigma \to \Sigma) \leftarrow \cala_G(P_M \to M) \rightarrow \cala_G(P_{\Sigma'} \to \Sigma')\text{.}
\end{equation}
We have to show that this defines a symmetric monoidal functor 
$\widetilde{\cala_G}: \cob^J(1,2,3) \to \Span$.

In particular, we have to show that the composition of morphisms is respected.

\begin{Lemma}\label{kleben}
Let $P_\Sigma \to \Sigma$ and $P_{\Sigma'} \to \Sigma'$ be two 1-morphisms in $\cob^J(1,2,3)$ 
which can be composed at the object  $P_S \to S$ to get
the 1-morphism 
$$P_\Sigma \circ P_{\Sigma'}:=
\big(P_\Sigma \sqcup_{P_S\times I} P_{\Sigma'} \to \Sigma \sqcup_{S \times I} \Sigma'\big)
\,\, ,$$
where $I = [0,1]$ is the standard interval. (Recall that
we are gluing over collars.)
Then the category $ \cala_G\big(P_\Sigma \circ P_{\Sigma'}\big)$ is the weak pullback of $\cala_G(P_\Sigma \to \Sigma)$ and $\cala_G(P_{\Sigma'} \to \Sigma')$ over $\cala_G(P_S \to S)$.
\end{Lemma}
\begin{proof}
By definition the category $$ \cala_G\big(P_\Sigma \circ P_{\Sigma'}\big)$$ has as objects twisted $G$-bundles over the 2-manifold 
$\Sigma \sqcup_{S \times I} \Sigma' =: N$. The manifold $N$ admits an open covering $N = U_0 \cup U_1$ with $U_0 = \Sigma \setminus S$ and $U_1 = \Sigma'\setminus S$ where the intersection is the cylinder 
$U_0 \cap U_1 = S \times (0,1)$. 
By construction, the restrictions of the glued bundle $P_N \to N$ to $U_0$ and $U_1$ are given by $P_\Sigma \setminus P_S$ and 
$P_{\Sigma'} \setminus P_S$. 

The natural inclusions $U_0 \to \Sigma$ and $U_1 \to \Sigma'$ induce equivalences
\begin{eqnarray*}
\cala_G(P_\Sigma \to \Sigma) \iso \cala_G( P_N|_{U_0} \to U_0) \\
\cala_G(P_{\Sigma'} \to \Sigma') \iso \cala_G( P_N|_{U_1} \to U_1)
\end{eqnarray*}
Analogously, we have an equivalence
$$
\cala_G\big(P_N|_{U_0 \cap U_1} \to U_0 \cap U_1\big) \iso 
\cala_G(P_S \to S) \,\,\, .
$$
At this point, we have reduced the claim to an
assertion about descent of twisted bundles which we will
prove in corollary \ref{desc_pullback}. 
This corollary implies that $\cala_G(P_N \to N)$ is 
the weak pullback of 
$\cala_G(P_N|_{U_0} \to U_0)$ and $\cala_G(P_N|_{U_1} \to U_1)$ over $\cala_G\big(P_N|_{U_0 \cap U_1}\big)$. Since weak pullbacks are invariant under 
equivalence of groupoids, this shows the claim.
\end{proof}

We now turn to the promised results about descent of
twisted bundles.
Let $P \to M$ be a $J$-cover over a manifold $M$ and 
$\{U_\alpha\}$ be an open covering of $M$, where for the 
sake of generality we allow for arbitrary open coverings. 
We want to show that twisted bundles can be glued together 
like ordinary bundles; while the precise meaning of this statement
is straightforward, we briefly summarize the relevant
definitions for the sake of completeness:
\begin{Definition}
Let $P \to M$ be a $J$-cover over a manifold $M$ and $\{U_\alpha\}$ be an 
open covering of $M$. The descent category $\Desc(U_\alpha,P)$ has 
\begin{itemize}
\item \underline{Objects:} families of $P|_{U_\alpha}$-twisted bundles 
$Q_\alpha$ over $U_\alpha$,  together with isomorphisms of twisted bundles
  $\varphi_{\alpha\beta} : Q_\alpha|_{U_\alpha \cap U_\beta} \iso Q_\beta|_{U_\alpha \cap U_\beta}$ satisfying the cocycle condition $\varphi_{\alpha\beta} \circ \varphi_{\beta\gamma} = \varphi_{\alpha\gamma}$.
\item \underline{Morphisms:}
families of morphisms $f_\alpha: Q_\alpha \to Q'_\alpha$ of twisted 
bundles such that over $U_{\alpha\beta}$ we have 
$\varphi'_{\alpha\beta} \circ (f_\alpha)|_{U_{\alpha\beta}} = (f_\beta)|_{U_{\alpha\beta}} \circ \varphi_{\alpha\beta}$.
\end{itemize}
\end{Definition}

\begin{Proposition}[Descent for twisted bundles]
Let $P \to M$ be a $J$-cover over a manifold $M$ and 
$\{U_\alpha\}$ be an open covering of $M$. Then the groupoid $\cala_G(P \to M)$ 
is equivalent to the descent category $\Desc(U_\alpha,P)$. 
\end{Proposition}
\begin{proof}
Note that the corresponding statements are true for 
$H$-bundles and for $J$-covers. Then the description in definition \ref{def:twisted}
of a twisted bundle as an $H$-bundle together with a 
morphism of the associated $J$-cover
immediately implies the claim.
\end{proof}

\begin{corollary}\label{desc_pullback}
For an open covering of $M$ by two open sets $U_0$ and $U_1$ the category $\cala_G(P \to M)$ is the weak pullback of $\cala_G(P|_{U_0}\to U_0)$ and $\cala_G(P|_{U_1}\to U_1)$ over $\cala_G(P|_{U_0 \cap U_1}\to U_0 \cap U_1)$.
\end{corollary}

In order to prove that the assignment \eqref{span1} and \eqref{span2} really promotes $\cala_G$ to a symmetric monoidal functor $\widetilde{\cala_G}: \cob^J(1,2,3) \to \Span$, it remains to show that $\cala_G$ preserves the monoidal structure. 

Now a bundle over a disjoint union is given by a pair of 
bundles over each component. Thus, for a disjoint union 
of $J$-manifolds $P \to M = (P_1 \sqcup P_2) 
\to (M_1 \sqcup M_2)$, we have 
$\cala_G(P \to M) \cong \cala_G(P_1 \to M_1) \times 
\cala_G(P_2 \to M_2)$. Note that the manifolds
$M,M_1$ and $M_2$ can also be cobordisms.
The isomorphism of categories is clearly associative and 
preserves the symmetric structure. Together 
with lemma \ref{kleben}, this proves the next proposition.

\begin{Proposition}
$\cala_G$ induces a symmetric monoidal functor 
$$\widetilde{\cala_G}: \cob^J(1,2,3) \to \Span$$
which assigns the spans \eqref{span1} and \eqref{span2} to 2 and 3-dimensional cobordisms with $J$-cover.
\end{Proposition}

\subsection{Twisted sectors and fusion}\label{twisted_sectors_and_fusion}

We next proceed to evaluate the
$J$-equivariant TFT $Z_G^J$ constructed in the last section
on the circle, as we did in section \ref{evaluation}
for the non-equivariant TFT.
We recall from section \ref{sec:twisted} the fact
that over the circle $\SS$ we have for each $j \in J$ 
a standard cover $P_j$. The associated  category
$$ \calc(G)_j := Z_G^J\big(P_j \to \SS\big) $$
is called the $j$-twisted sector of the theory; the sector 
$\calc(G)_1$ is called the neutral sector. By lemma
\ref{trivbundle}, we have an equivalence
$\cala_G(P_1\to \SS)\cong \cala_G(\SS)$; hence we
get an equivalence of categories  $\calc(G)_1\cong\calc(G)$,
where $\calc(G)$ is the category arising in  the
non-equivariant Dijkgraaf-Witten model, we discussed
in section \ref{evaluation}.  We
have already computed the twisted sectors as
abelian categories in example 
\ref{ex_twisted} and note the result for future reference:

\begin{Proposition}\label{sectors}
For the $j$-twisted sector of equivariant Dijkgraaf-Witten theory,
we have an equivalence of abelian categories
$$\calc(G)_j \cong [H_j //G,\vect_\KK] \,\, , $$
where $H_j //G$ is the action groupoid given by the conjugation action of $G$ on $H_j := \pi^{-1}(j)$. More concretely, the
category $\calc(G)_j$ is equivalent to the category of $H_j$-graded 
vector spaces $V = \bigoplus_{h \in H_j} V_h$ together with a 
$G$-action on $V$ such that
$$ g.V_h \subset V_{ghg^{-1}}\text{.}$$
\end{Proposition}

As a next step, we want to make explicit additional 
structure on the categories $\calc(G)_j$ coming from certain 
cobordisms. Therefore, consider the pair of pants $\pop$:
\begin{center}
\includegraphics{Hose.pdf}
\end{center}
The fundamental group of $\pop$ is the free group on two 
generators. Thus, given a pair of group elements
$j,k\in J$, there is a $J$-cover $P^\pop_{j,k} \to \pop$ 
which restricts to the standard covers $P_j$ and $P_k$ 
on the two ingoing boundaries and to the standard cover
$P_{jk}$ on the outgoing boundary circle. (To find a
concrete construction, one should fix a parametrization of 
the pair of pants $\pop$.)
The cobordism $P^\pop_{j,k}$ is a morphism
\begin{equation} \label{fusionproduct}
P^\pop_{j,k}: \big(P_j \to \SS\big) \sqcup \big(P_k \to 
\SS\big) \longrightarrow \big(P_{jk} \to \SS\big)
\end{equation}
in the category $\cob^J(1,2,3)$. Applying the equivariant
TFT-functor $Z_G^J$ yields a functor
$$ \otimes_{jk} :\quad  \calc(G)_j \boxtimes \calc(G)_k \longrightarrow \calc(G)_{jk}.$$
We describe this functor in terms of the equivalent categories
of graded vector spaces as a functor
$$H_j//G\mod \times H_k//G\mod \to H_{jk}// G\mod \,\, .$$

\begin{Proposition}\label{prop:fusionproduct}
For $V= \bigoplus_{h\in H_j} V_h \in H_j//G\mod $ and 
$W = \bigoplus W_h \in  H_k//G\mod$ the product 
$V \otimes_{jk} W \in  H_{jk}//G\mod$ is given by
$$ (V \otimes_{jk} W)_{h} = \bigoplus_{st = h} V_s \otimes W_t $$
together with the action $g .(v \otimes w) = g.v \otimes g.w$.
\end{Proposition}
\begin{proof}
As a first step we have to compute the span 
$\widetilde{\cala_G}(P^\pop_{j,k})$
associated to the cobordism $P^H_{j,k}$. From 
the description of twisted bundles in proposition 
\ref{covering} and the fact that the fundamental group of $\pop$ is the 
free group on two generators, we derive the following
equivalence of categories: 
$$ \cala_G\big(P^\pop_{jk} \to \pop\big) \cong (H_j \times H_k) // G
\,\,\, . $$
Here we have
$H_j \times H_k = \{(h,h') \in H \times H \mid \pi(h) 
= j,~ \pi(h') = k\}$, on which $G$ acts by simultaneous
conjugation.
This leads to the span of action groupoids
$$ H_j //G \times H_k //G \longleftarrow (H_j \times H_k) // G 
\longrightarrow H_{jk} //G $$
where the left map is given be projection to the factors 
and the right hand map by multiplication. Applying the 
2-linearization functor
$\widetilde{ \calv_\KK }$ from proposition \ref{lin2}
amounts to computing the corresponding pull-push 
functor. This yields the result.
\end{proof}

Next, we consider the 2-manifold $\cyl$ given by the cylinder 
over $\SS$, i.e. $\cyl = \SS \times I$:
\begin{center}
\includegraphics{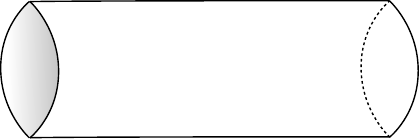}
\end{center}
There exists a cover $P^\cyl_{j,x} \to \cyl$ for $j,x \in J$ 
that restricts to $P_{j}$ on the ingoing circle and to 
$P_{xjx^{-1}}$ on the outgoing circle. The simplest
way to construct such a cover is to consider the
cylinder $P_{xjx^{-1}} \times I \to \SS \times I$ and 
to use the identification of
$P^\cyl_{j,x}$ over (a collaring neighborhood of) the outgoing 
circle by the identity and over the ingoing circle 
the identification by the morphism
$P_\cyl|_{\SS \times {1}} = P_j \to P_{x^{-1}jx}$ given by conjugation with $x$. In this way, we obtain a cobordism that
is a 1-morphism
\begin{equation} \label{jaction}
  P^\cyl_{j,x}: (P_{j} \to \SS) \longrightarrow (P_{xjx^{-1} } \to \SS)
\end{equation}
in the category $\cob^J(1,2,3)$ and hence induces a functor
$$ \phi_x: \calc(G)_{j} \to \calc(G)_{xjx^{-1}}.$$
We compute the functor on the equivalent action groupoids
explicitly:

\begin{Proposition} \label{prop:action}
The image under $\phi_x$ of an object 
$V = \bigoplus V_h \in H_{j}//G\mod$ is
the graded vector space with homogeneous component
$$ \phi_x(V)_h = V_{s(x^{-1}) h s\left(x^{-1}\right)^{-1}} $$
for $h\in H_{xjx^{-1}}$ and with $G$-action on $v \in V_h$ 
given by $s(x^{-1}) g s\left(x^{-1}\right)^{-1} \cdot v$.
\end{Proposition}

\begin{proof}
As before we compute the span $\widetilde{\cala_G}(P^\cyl_{j,x})$. 
Using explicitly the equivalence given in the proof of proposition 
\ref{covering}, we obtain the span of action groupoids
$$ H_{j}//G \leftarrow H_{xjx^{-1}}//G \rightarrow H_{xjx^{-1}}//G $$
where the right-hand map is the identity and the left-hand map is given by 
$$ (h,g) \mapsto \big(s(x^{-1}) h s(x^{-1})^{-1}, s(x^{-1}) g s(x^{-1})^{-1}\big)
\,\, . $$
Computing the corresponding pull-push functor, which here in fact only consists of a pullback, 
shows the claim.
\end{proof}

Finally we come to the structure corresponding to the 
braiding of section \ref{evaluation}. Note that the 
cobordism that interchanges the two ingoing circles of 
the pair of pants $\pop$, as in the following picture,
\begin{center}
\includegraphics{Dose.pdf}
\end{center}
can also be realized as the diffeomorphism $F: \pop \to \pop$ 
of the pair of pants
that rotates the ingoing circles counterclockwise around 
each other and leaves the outgoing circle fixed. 
In this picture, we think of the cobordism as the cylinder 
$\pop \times I$ where the identification with $\pop$ on the top 
is the identity and on the bottom is given by the
diffeomorphism $F$. More explicitly, denote by 
$\tau:\ \SS\times\SS\to\SS\times\SS$ the map that interchanges
the two copies. We then consider the following diagram
in the 2-category $\cob(1,2,3)$:
$$ \xymatrix{ 
&&\pop\ar@{=>}^F[dd]&& \\
\SS\times \SS\ar_\tau[dr]\ar^\iota[urr]&&&&\ar[ull]\ar[dl]\SS \\
&\SS\times\SS\ar^\iota[rr] &\ &\pop 
}$$
where $\iota:\SS\times\SS\to \pop$ is the standard inclusion of
the two ingoing boundary circles into the trinion $\pop$.

Our next task is to lift this situation to manifolds
with $J$-covers. On the ingoing trinion, we take the $J$ cover
$P^\pop_{jk}$. We denote the symmetry isomorphism in $\cob^J(1,2,3)$
by $\tau$ as well. Applying the diffeomorphism of the
trinion explicitly, one sees that the outgoing trinion will 
have monodromies $jkj^{-1}$ and $j$ on the ingoing circles.
Hence we have to apply a $J$-cover $P_{j,k}^\cyl$
of the cylinder $\cyl$ first to one insertion. 
The next lemma asserts that then the 2-morphism
in $\cob^J(1,2,3)$ is fixed:

\begin{Lemma}\label{braiding-morphism}
In the 2-category $\cob^J(1,2,3)$, there 
is a unique 2-morphism
$$\hat F:\quad P^\pop_{j,k} \Longrightarrow \big(P^\pop_{jkj^{-1},j}\big) 
\circ \tau \circ \big(\id \sqcup P^\cyl_{j,k}\big) $$
that covers the 2-morphism $F$ in $\cob(1,2,3)$.
\end{Lemma}
\begin{proof}
First we show that a morphism 
$\tilde F: P^\pop_{jkj^{-1},j} \to P^\pop_{j,k}$
can be found that covers the diffeomorphism $F:  \pop \to \pop$.
This morphism is most easily described using the 
action of $F$ on the fundamental group $\pi_1(\pop)$ of
the pair of pants. The latter is a free group with two
generators which can be chosen as the paths $a,b$ around the
two ingoing circles,  
$\pi_1(\pop) = \mathbb{Z}* \mathbb{Z} = \langle a,b \rangle$. 
Then the induced action of $F$\ on the generators is
$\pi_1(F)(a) = aba^{-1}$ and $\pi_1(F)(b) = a$. 
Hence, we find on the covers $F^*P_{j,k} \cong P_{jkj^{-1},j}$.
This implies that we have a diffeomorphism
$\tilde F: P_{jkj^{-1},j} \to P_{j,k}$ covering $F$.

To extend $\tilde F$ to a 2-morphism in $\cob^J(1,2,3)$, we have 
to be a bit careful about how we consider 
the cover $P^\pop_{jkj^{-1},j} \to \pop$ of the trinion as a 
1-morphism. In fact, it has to be considered as a morphism
$(P_j \to \SS) \sqcup (P_k \to \SS) \longrightarrow  (P_{jk} \to \SS)$
where the ingoing components are first exchanged and then the 
identification of $P_k \to \SS$ and $P_{jkj^{-1}} \to \SS$ 
via the conjugation isomorphisms $P^\cyl_{j,k}$ induced by
covers of the cylinders is
used first, compare the lower arrows in the preceding
commuting diagram. This yields the composition
$ \big(P^\pop_{jkj^{-1},j}\big) \circ \tau \circ 
\big(\id \sqcup P^\cyl_{j,k}\big) $
on the right hand side of the diagram.
\end{proof}

The next step is to apply the TFT functor $Z_G^J$ to 
the 2-morphism $\hat F$. The target 1-morphism of
$\hat F$ can be computed using the fact that  $Z_G^J$ 
is a symmetric monoidal 2-functor; we find the
following functor $\calc(G)_j\otimes \calc(G)_k
\to \calc(G)_{jk}$:
$$Z_G^J\Big(\big(P^\pop_{jkj^{-1},j}\big) \circ \tau 
\circ \big(\id \sqcup P^\cyl_{j,k}\big)\Big) 
= (-)^j \otimes_{jkj^{-1},j}^{op} (-)$$
We thus have the functor which acts on objects as
$(V,W) \mapsto \phi_j(W) \otimes V$
for $V \in \calc(G)_j$ and $W \in \calc(G)_k$.

Then $c := Z_G^J(\hat F)$ is a natural transformation
$ (-) \otimes_{j,k}(-) \Longrightarrow  (-)^j
\otimes_{jkj^{-1},j}^{op}(-) $
i.e. a family of isomorphisms
\begin{equation}
  c_{V,W}:V \otimes_{j,k} W \iso \quad \phi_j(W) \otimes_{jkj^{-1},j} V
\end{equation}
in $\calc(G)_{jk}$ for $V \in \calc(G)_j$ and $W \in \calc(G)_k$.

We next show how this natural transformation is expressed
when we use the equivalent description of the categories
$\calc(G)_j$ as vector bundles on action groupoids:

\begin{Proposition}\label{prop:braiding}
For $V= \bigoplus V_h \in H_j//G\mod$ and 
$W= \bigoplus W_h \in H_k//G\mod$ the natural isomorphism 
$c_{V,W}: V \otimes W \to \phi_j(W) \otimes V$ is given by
$$ v \otimes w \mapsto (s(j^{-1})h).w \otimes v$$
for $v  \in V_h$ with $h \in H_{j}$ and $w \in W$.
\end{Proposition}
\begin{proof}
We first compute the 1-morphism in the category $\Span$ 
of spans of finite groupoids that corresponds to
the target 1-morphism $\big(P^\pop_{jkj^{-1},j}\big) \circ \tau 
\circ \big(id \sqcup P^\cyl_{j,k}\big)$. From the previous
proposition,  we obtain the following zig-zag diagram:
$$ H_j//G \times H_k //G \leftarrow H_{jkj^{-1}}//G \times H_j//G
\leftarrow (H_{jkj^{-1}} \times H_k) //G \to H_{jk} //G
\,\, . $$
The first morphism is given by the morphisms implementing
the $J$-action that has been computed in the 
proof of proposition \ref{prop:action}, composed 
with the exchange of factors. The second 1-morphism
is obtained from the two projections and the last 1-morphism 
is the product in the group $H$. 

Thus, the 2-morphism $\hat F$ from Lemma \ref{braiding-morphism} yields a 2-morphism $\hat F_G$ 
in the diagram
$$\xymatrix{
& & H_j \times H_k //G \ar[lld]\ar[rrd]\ar@{=>}[dd]^{\hat F_G} & \\
  H_j//G \times H_k //G & & & & H_{jk} //G \\
& H_{jkj^{-1}}//G \times H_j//G\ar[lu] & (H_{jkj^{-1}} \times H_j) //G \ar[l]\ar[rru] &
}$$
where $\hat F_G$ is induced by the equivariant map 
$(h,h') \mapsto (h h' h^{-1},h)$. Once the situation
is presented in this way, one can carry out explicitly
the calculation along the lines described in
\cite[Section 4.3]{Morton} and obtain the result.
\end{proof}

A similar discussion can in principle be carried out to compute the
associators. More generally, structural morphisms 
on $H//G\mod$ can be derived from suitable 3-cobordisms. The relevant
computations become rather involved. On the other hand,
the category $H//G\mod$ also inherits structural morphisms
from the underlying category of vector spaces. We 
will use in the sequel the latter type of structural
morphism.

\section{Equivariant Drinfel'd double}

The goal of this section is to show that the category
$\cdw := \bigoplus_{j \in J} \calc(G)_j$
comprising the categories 
we have constructed in proposition \ref{sectors}
has a natural structure of a $J$-modular category.

Very much like ordinary modularity, $J$-modularity
is a completeness requirement for the relevant
tensor category that is suggested by
principles of field theory. Indeed,
it ensures that one can construct a $J$-equivariant
topological field theory, see \cite{turaev2010}.
For the definition of 
$J$-modularity we refer to  \cite[Definition 10.1]{kirI17}.

To establish the structure of a modular tensor category
on the category found in the previous sections, we 
realize this category
as the representation category of a finite-dimensional
algebra, more precisely of a $J$-Hopf algebra.
This section is organized as follows: we first
recall the notions of equivariant ribbon categories and
of equivariant ribbon algebras, taking into account
a suitable form of weak actions. In section 4.3,
we then present the appropriate generalization of the
Drinfel'd double that describes the category $\cdw$.
We then describe its orbifold category as the
category of representations of a braided Hopf algebra,
which allows us to establish the modularity of the orbifold
category. We then apply a result of
\cite{kirI17} to deduce that the structure with which
we have endowed $\cdw$ is the one of a $J$-modular
tensor category.

The Hopf algebraic structures endowed with weak actions
we introduce in this section might be of independent interest.

\subsection{Equivariant braided categories.}

Let $1\to G\to H\stackrel\pi\to J\to 1$ be an exact sequence of finite
groups. The normal subgroup $G$ acts on $H$ by
conjugation; denote by $H//G$ the corresponding action 
groupoid. We consider the
functor category $H//G\mod:=[H//G,\vect_\KK]$, where
$\KK$ is an algebraically closed field of characteristic
zero. The category $H//G\mod$ is
the category of $H$-graded vector spaces,
endowed with an action of the subgroup $G$ such that 
$g.V_h \subset V_{ghg^{-1}}$ for all $g\in G, h\in H$.

An immediate corollary of proposition \ref{sectors}
is the following description of the category
$\cdw := \bigoplus_{j \in J} \calc(G)_j$ 
as an abelian category:

\begin{Proposition}\label{HG-Mod}
The category $\cdw$ is equivalent, as an
abelian category, to the category $H//G\mod$.
In particular, the category $\calc^J(G)$ is a 2-vector space
in the sense of definition \ref{twovect}.
\end{Proposition}

\begin{proof}
With $H_j:=\pi^{-1}(j)$, proposition \ref{sectors}
gives the equivalence $\calc(G)_j\cong\H_j//G\mod$ 
of abelian categories. The equivalence of categories
$\cdw\cong H//G\mod$ now follows
from the decomposition $H = \bigsqcup_{j \in J} H_j$. 
By \cite[Lemma 4.1.1]{Morton08}, the representation 
category of a finite groupoid is a  2-vector space.
\end{proof}

Representation categories of finite groupoids are very
close in structure to representation categories of finite
groups. In particular, there is a complete character
theory that describes the simple objects, 
see appendix  \ref{groupoidrep}.

We next introduce equivariant categories.

\begin{Definition}
Let $J$ be a finite group and $\calc$ a category. 
\begin{enumerate}
\item
A \emph{categorical action} of the group $J$ on 
the category $\calc$ consists of the
following data:
\begin{itemize}
\item[--] A functor $\phi_j:\calc \to \calc$ for 
every group element $j \in J$.
\item[--] A functorial isomorphism 
$\alpha_{i,j}:\phi_i \circ \phi_j \stackrel\sim\to \phi_{ij}$, called \emph{compositors}, for every 
pair of group elements $i,j \in J$
\end{itemize} such that the coherence conditions
\begin{equation}\label{coherence-compositor}
\alpha_{ij,k} \circ \alpha_{i,j}  = \alpha_{i,jk} 
\circ \phi_i(\alpha_{j,k}) \quad \text{ and } \quad \phi_1 = \id
\end{equation}
hold.

\item
If $\calc$ is a monoidal category, we only consider
actions by monoidal functors $\phi_j$ and 
require the natural
transformations to be monoidal natural transformations.
In particular, for  each group element $j\in J$,
we have the additional
datum of a natural isomorphism
$$\gamma_j(U,V): \phi_j(U)\otimes\phi_j(V)\stackrel\sim\to\phi_j(U\otimes V)$$
for each pair of objects $U,V$ of $\calc$ such that
the following diagrams commute: 
$$ \xymatrix{
^{jk}X\otimes {} ^{jk}Y \ar^{\gamma_{jk}(X,Y)}[rrrr]
\ar_{\alpha_{jk}(X)\otimes\alpha_{jk}(Y)}[d]&&&&
^{jk}(X\otimes Y)\ar^{\alpha_{jk}(X\otimes Y)}[d]\\
^j(^k{} (X))\otimes {}^j(^k{} ( Y))
\ar_{^j\gamma_k(X,Y)\circ\gamma_j(^kX,^k Y)}[rrrr]
&&&&
^j(^k{} (X\otimes Y))
}$$

(The data of a monoidal functor includes an isomorphism $\phi_j(1)\to 1$ in principal, but in this paper, the isomorphism will be the identity and therefore we will suppress it in our discussion.)
We use the notation $^jU := \phi_j(U)$ for the image of
an object $U \in \calc$ under the functor $\phi_j$.

\item 
A \emph{$J$-equivariant} category $\calc$ is a category
with a decomposition $\calc = \bigoplus_{j \in J} \calc_j$ 
and a categorical action of $J$, subject to the
compatibility requirement
\[\phi_{i}\calc_j \subset \calc_{iji^{-1}}\]
with the grading.

\item
A \emph{$J$-equivariant tensor category} is a $J$-equivariant 
monoidal category $\calc$, subject 
to the compatibility requirement that
the tensor product of two homogeneous elements 
$U \in \calc_i, V \in \calc_j$ is again homogeneous, 
$U \otimes V \in \calc_{ij}$.

\end{enumerate}
\end{Definition}

\begin{Remark}
We remark that the condition $\phi_1 = \id$ \ref{coherence-compositor} should in general be replaced by an extra datum, an isomorphism $\eta:\id \stackrel\sim\to \phi_1$ and two coherence conditions which involve the compositors $\alpha_{i,j}$. The diagrams can be found as follows:
For any category $\calc$,
consider the category $\AUT(\calc)$ whose objects
are automorphisms of $\calc$ and whose morphisms are 
natural isomorphisms. 
The composition of functors and natural transformations 
endow $\AUT(\calc)$ with the natural structure of a strict 
tensor category.
A categorical action of a finite group $J$ on a category
$\calc$ then amounts to a tensor functor 
$\phi:J \to \AUT(\calc)$, where $J$ is seen as a tensor category
with only identity morphisms, compare also remark
\ref{rem:3.2}.3.
The condition $\phi_1 = \id$ holds in the categories we are interested in, hence we impose it.

Similarly, we consider
for a monoidal category $\calc$ the category
$\AUT mon(\calc)$ whose objects are monoidal 
automorphisms of $\calc$ and whose morphisms are monoidal
natural automorphisms. The categorical actions we consider
for monoidal categories are then tensor
functors $\phi:J \to \AUT mon(\calc)$
 For more details, we
refer to \cite{turaev2010} Appendix 5.
\end{Remark}

The category $H//G\mod$ has a natural structure of a
monoidal category: the tensor product of two objects
$V=\oplus_{h\in H} V_h$ and $W=\oplus_{h\in H} W_h$
is the vector space $V\otimes W$ with $H$ grading
given by $(V\otimes W)_h := \oplus_{h_1h_2=h}
V_{h_1}\otimes W_{h_2}$ and $G$ action given by
$g.(v\otimes w)= g.v\otimes g.w$. The associators
are inherited from the underlying category of
vector spaces.

\begin{Proposition}
Consider an exact sequence of groups
$1\to G\to H \to J\to 1$. Any choice of a 
a set-theoretic section $s: J\to H$ allows us to endow the
abelian category $H//G\mod$ with the structure of a
$J$-equivariant tensor category as follows: the functor
$\phi_j$ is given by shifting the grading from
$h$ to $s(j)hs(j)^{-1}$ and replacing the action by
$g$ by the action of $s(j)gs(j)^{-1}$.
The isomorphism 
$\alpha_{i,j}:\phi_i\circ\phi_k\to\phi_{ij}$ 
is given by the left action action of the element
$$ \alpha_{i,j} = s(i)s(j)s(ij)^{-1} \,\,.$$
\end{Proposition}

The fact that the action is only a weak action
thus accounts for the failure of $s$ to be a section in the
category of groups.

\begin{proof}
Only the coherence conditions 
$\alpha_{ij,k} \circ \alpha_{i,j}  = \alpha_{i,jk} 
\circ \phi_i(\alpha_{j,k})$ remain to be checked. 
By the results of Dedecker and Schreier, cf.\ proposition
\ref{dedeschreier}, the group elements
$s(i)s(j)s(ij)^{-1}\in G$ are the coherence cells of 
a weak group action of $J$ on $H$. By
definition \ref{Def:action}, this implies the coherence
identities, once one takes into account that that 
composition of functors is written in different order 
than group multiplication.
\end{proof}

We have derived in section 
\ref{twisted_sectors_and_fusion} from the geometry
of extended cobordism categories
more structure on the geometric category $\cgj$.
In particular, we collect the functors $\otimes_{jk}: \calc(G)_j 
\boxtimes \calc(G)_k \to \calc(G)_{jk}$ from  
proposition \ref{prop:fusionproduct}  into a functor 
\begin{equation}
 \otimes : \calc \boxtimes \calc \to \calc\text{.}
\label{eq:grouptogether}
\end{equation}
Another structure are the isomorphisms  
$V \otimes W \to \phi_j(W) \otimes V$ 
for $V \in \calc(G)_j$, described in proposition 
\ref{prop:braiding}. Together with the associators,
this suggests to endow the category $\cdw$ 
with a structure of a braided $J$-equivariant tensor 
category: 

\begin{Definition}\label{J-equ.tc}

A \emph{braiding} on a $J$-equivariant tensor category is 
a family
\[c_{U,V}:U\otimes V \to {}^jV \otimes U\]
of isomorphisms, one for every pair of objects $U \in \calc_i, V \in \calc_j$,
which are natural in $U$ and $V$. 
 Moreover, a braiding is required to satisfy
an analogue of the hexagon axioms (see 
\cite[appendix A5]{turaev2010})
and to be preserved under the action of $J$,
i.e.\ the following diagram commutes for all objects
$U,V$ with $U\in\calc_j$ and $i\in J$
\begin{equation}\label{action-braiding-compatibility}
\xymatrix{
^i(U\otimes V)\ar_{\gamma_i}[d]\ar^{^i(c_{U,V})}[rr]
&& ^i(^jV \otimes U) \ar^{\gamma_i}[rr]&& ^i(^jV)\otimes \, ^iU 
\ar^{\alpha_{ij}(V)\otimes\id}[d]\\
^iU\otimes ^iV\ar_{c_{^iU,^iV}}[rr] && 
^{ij^{-1}}(^iV)\otimes\, ^iU 
\ar_{\alpha_{iji^{-1},i(V)}\otimes\id}[rr]
&& ^{ij}V\otimes \, ^i U
}
\end{equation}
\end{Definition}

\begin{Remark}
\begin{enumerate}
\item
It should be appreciated that a braided $J$-equivariant
category is not, in general, a braided 
category. Its {\em neutral component} $\calc_1$ with
$1\in J$ the neutral element, is a braided tensor category.

\item
By replacing the underlying category by an equivalent
category, one can replace a weak action by a strict
action, compare \cite[Appendix A5]{turaev2010}.
In our case, weak actions actually lead to simpler
algebraic structures.

\item
The $J$-equivariant monoidal category $H//G\mod$ has a
natural braiding isomorphism that has been
described in proposition \ref{prop:braiding}

\end{enumerate}
\end{Remark}

We use the equivalence of abelian categories between
$\cgj$ and $H//G\mod$ to endow the category
$\cgj$ with associators. The category has now enough
structure that we can state our next result:

\begin{Proposition}\label{prop:Jtensor}
The category $\cgj$, with the tensor product functor
from (\ref{eq:grouptogether}), can be endowed with
the structure of a braided $J$-equivariant tensor 
category such that the isomorphism $\cgj
\cong H//G\mod$ becomes an isomorphism of
braided $J$-equivariant tensor categories.
\end{Proposition}
\begin{proof}

The compatibility with the grading is implemented by definition via the graded components $\otimes_{jk}$ of $\otimes$ and the graded components of $c_{V,W}$. It remains to check that the action is by tensor functors and that the braiding satisfies the hexagon axiom. The second boils down to a simple calculation and the first is seen by noting that the action is essentially an index shift which is preserved by tensoring together the respective components. 
\end{proof}

\subsection{Equivariant ribbon algebras}

In the following, let $J$ again be a finite group.
To identify the structure of a $J$-modular tensor
category on the geometric category $\cgj$, we need
dualities. This will lead us
to the discussion of (equivariant) ribbon algebras.
Apart from strictness issues, our discussion closely follows
\cite{turaev2010}. 
We start our discussion with the relevant category-theoretic
structures, which generalize premodular categories (cf. definition \ref{modularity}) to the equivariant setting.

\begin{Definition}\label{ribbon_def}
\begin{enumerate}
\item
A \emph{$J$-equivariant ribbon category} is a $J$-braided 
category with dualities and a family of isomorphisms 
$\theta_V:V \to{} ^jV$ for all $j \in J, V \in \calc_j$, 
such that $\theta$ is compatible with duality and the action of $J$ (see \cite[VI.2.3]{turaev2010} for the identities).
In contrast to \cite{turaev2010}, we allow weak $J$-actions
and thus require the diagram
\begin{align}\label{com.diagram-twist}
\begin{xy}
  \xymatrix{
      U\otimes V \ar[rrrr]^{\theta_{U\otimes V}} \ar[d]_{\theta_U \otimes \theta_V}    &&&&   ^{ij}(U\otimes V) \\
   ^iU\otimes ^jV \ar[rd]_{c_{{}^iU, ^jV}}             & &&&  ^{iji^{-1}}(^i U) \otimes ^{ij}V\ar[u]_{\gamma_{ij} \circ (\alpha_{iji^{-1},i} \otimes \id)} \\
    &^i(^jV) \otimes {}^i U\ar[rr]_{\alpha_{i,j}\otimes \id} & &^{ij}V \otimes {}^i U\ar[ur]_{c_{{}^{ij}V,^{i}U}} &
  }
\end{xy}
\end{align}
involving compositors,
to commute for $U \in \calc_i$ and $V \in \calc_j$.

\item A \emph{$J$-premodular category} is a $\KK$-linear, abelian, finitely semi-simple $J$-equivariant ribbon category such that the tensor product is a $\KK$-bilinear functor and the tensor unit is absolutely simple. 
\end{enumerate}
\end{Definition}

\begin{Remark}
The following facts directly follow from the
definition
of $J$-equivariant ribbon category:
the neutral component $\calc_1$ is itself a braided
tensor category. In particular, it contains the
tensor unit of the $J$-equivariant tensor category.
The dual object of an object $V\in \calc_j$ is in 
the category $\calc_{j^{-1}}$.
\end{Remark}

We will not be able to directly endow the geometric
category $\cgj$ with the structure of a $J$-equivariant ribbon category. 
Rather, we will realize an equivalent category as the
category of modules over a suitable algebra.
To this end, we introduce in several steps the notions of a
$J$-ribbon algebra and analyze the extra structure 
induced on its representation category.

\begin{Definition}
Let $A$ be an (associative, unital) algebra over a
field $\KK$. A \emph{weak $J$-action} on $A$ 
consists of algebra automorphisms $\varphi_j \in \Aut(A)$,
one for every element $j\in J$, and invertible 
elements $c_{ij} \in A$, one for every pair of elements 
$i,j \in J$, such that for all $i,j,k \in J$
the following conditions hold:
\begin{equation}\label{coherence-compositors2}
  \varphi_i \circ \varphi_j = \Inn_{c_{i,j}} \circ \varphi_{ij}\qquad
  \varphi_i(c_{j,k}) \cdot c_{i,jk}  = c_{i,j} \cdot 
  c_{ij,k} \quad \text{ and } \quad c_{1,1} = 1
\end{equation}
Here $\Inn_x$ with $x$ an invertible element of $A$ denotes 
the algebra automorphism $a \mapsto xax^{-1}$.
A weak action of a group $J$ is called
strict, if  $c_{i,j} = 1$ for all pairs $i,j\in J$. 
\end{Definition}

\begin{Remark}
As discussed for weak actions on groups in remark
\ref{rem:3.2}, a weak action 
on a $\KK$-algebra $A$ can be seen as a categorical action 
on the category
which has one object and the elements 
of $A$ as endomorphisms.
\end{Remark}

We now want to relate a weak action $(\varphi_j,c_{i,j})$
of a group $J$ on an algebra $A$ to a categorical action 
on the representation category $A\mod$. To this end, 
we define for each element $j\in J$ a functor
on objects by
$$^j(M,\rho) := (M,\rho \circ (\varphi_{j^{-1}}\otimes \id_M))$$
and on morphisms by $^jf = f$. For the functorial
isomorphisms, we take
\begin{equation*}
\alpha_{i,j}(M,\rho):= \rho\big((c_{j^{-1},i^{-1}})^{-1}\otimes \id_M\big).
\end{equation*}
The inversions in the above formulas make sure that the action on the level of categories really becomes a left action.

\begin{Lemma}\label{action}
Given a weak action of $J$ on a $\KK$-algebra $A$,
these data define a categorical action on 
the category $A\mod$.
\end{Lemma}

\begin{proof}
Let $V$ be an $A$-Modul. We first show that $\alpha_{i,j}$ is a morphism $^{i}(^{j}V)\to ^{ij}V$ in $A\mod$:
Let $V\in A\mod$, then for every $v \in V, a \in A$, we have:
\begin{align*}
\alpha_{i,j}(V) \circ \rho_{^{i}(^jV)}(a\otimes v)& = (c_{i^{-1},j^{-1}})^{-1}\varphi_{j^{-1}} \circ \varphi_{i^{-1}}(a).v\\ 
& = \varphi_{(ij)^{-1}}(a)(c_{i^{-1},j^{-1}})^{-1}.v\\
& = \rho_{(^{ij}V)} \circ (\id_A \otimes \alpha_{i,j}(V))(a \otimes v),
\end{align*}
where we used the abbreviation $a.v := \rho(a\otimes v)$.
The validity of the coherence condition \eqref{coherence-compositor} for the $\alpha_{i,j}$ follows from the second equation in \eqref{coherence-compositors2}, since the right side of \eqref{coherence-compositor} evaluated on an element $v \in V$ reads
\[ \alpha_{i,jk} \circ \phi_i(\alpha_{j,k}) (v)  =   (c_{(jk)^{-1},i^{-1}})^{-1}(c_{k^{-1},j^{-1}})^{-1}.v \]
and the left one is:
\[\alpha_{ij,k} \circ \alpha_{i,j}(v) =  (c_{k^{-1},(ij)^{-1}})^{-1} \varphi_{k^{-1}} (c_{j^{-1},i^{-1}})^{-1} .v\]
\end{proof}

We next turn to an algebraic structure that yields
$J$-equivariant tensor categories.

\begin{Definition}\label{JHopf}
A \emph{$J$-Hopf algebra} over $\KK$ is a Hopf algebra 
$A$ with a $J$-grading $A = \bigoplus_{j\in J} A_{j}$ 
and a weak $J$-action such that:
\begin{itemize}
\item 
The algebra structure of $A$ restricts to
the structure of an associative algebra on
each homogeneous component so that $A$ is the direct 
sum of the components $A_j$ as an algebra.

\item
$J$ acts by homomorphisms of Hopf algebras.

\item The action of $J$ is compatible with the grading,
i.e.\ $\varphi_i(A_j) \subset A_{iji^{-1}}$

\item 
The coproduct $\Delta:\ A\to A\otimes A$ respects
the grading, i.e.\
$$\ \Delta (A_j) \subset \bigoplus_{p,q\in J,pq=j} A_p\otimes
A_q \,\, . $$

\item The elements $(c_{i,j})_{i,j\in J}$ are group-like, i.e $\Delta(c_{i,j}) = c_{i,j}\otimes c_{i,j}$.
\end{itemize}
\end{Definition}

\begin{Remark}
\begin{enumerate}
\item
For the counit $\epsilon$ and the antipode $S$ of
a $J$-Hopf algebra, the compatibility relations
with the grading $\epsilon(A_j) = 0$ for $j \ne 1$ and 
$S(A_j) \subset A_{j^{-1}}$ are immediate consequences of
the definitions.

\item The restrictions of the structure maps
endow the homogeneous component $A_1$ of $A$
with the structure of a Hopf algebra with a weak
$J$-action.

\item $J$-Hopf algebras with strict $J$-action 
have been considered under the name ``$J$-crossed Hopf coalgebra''
in \cite[Chapter VII.1.2]{turaev2010}. 


\end{enumerate}
\end{Remark}

The category $A\mod$ of finite-dimensional modules
over a $J$-Hopf algebra inherits a natural duality
from the duality of the underlying category of 
$\KK$-vector spaces. The weak action described in
Lemma \ref{action} is even a monoidal action,
since $J$ acts by Hopf algebra morphisms.
A grading on $A\mod$ can be given by taking 
$(A\mod)_j = A_j\mod$ as the $j$-homogeneous component. 
From the properties of a $J$-Hopf algebra one can finally deduce 
that the tensor product, duality and grading are compatible with 
the $J$-action. We have thus arrived at the following statement:

\begin{Lemma}\label{J-equivariant}
The category of representations of a $J$-Hopf algebra has a
natural structure of a $\KK$-linear, abelian 
$J$-equivariant tensor category with compatible duality 
as introduced in definition \ref{J-equ.tc}.
\end{Lemma}

\begin{proof}
We show, that the grading and the action on $A\mod$ are compatible. Let $
V \in A_j\mod$, then the $j$-component $1_j$ of the unit in A acts as the identity on $V$. We have to check, that $^{i}V \in A_{iji^{-1}}\mod$, i.e that $1_{iji^{-1}}$ acts as the identity on $^iV$. For $v\in ^iV$, we have:
\[\rho_i(1_{iji^{-1}}\otimes v) = \varphi_{i^{-1}}(1_{iji^{-1}}).v = 1_j.v = v\] 
This shows $^i V \in A_{iji^{-1}}\mod$.
\end{proof}

The representation category of a braided Hopf algebra
is a braided tensor category. If the Hopf algebra has, moreover, a
twist element, its representation category is even
a ribbon category. We now present $J$-equivariant
generalizations of these structures. 

\begin{Definition}\label{def:jribbon}
 Let $A$ be a $J$-Hopf algebra. 
\begin{enumerate}
\item
A $J$-equivariant $R$-matrix is an element $R = R_{(1)}\otimes R_{(2)} \in A\otimes A$ such that for $V\in (A\mod)_i$, $W\in A\mod$, the map
\begin{align*}
c_{VW}:& V\otimes W \to ^{i}W \otimes V\\
&v\otimes w \mapsto R_{(2)}.w\otimes R_{(1)}.v
\end{align*}
is a $J$-braiding on the category $A\mod$ according to definition \ref{J-equ.tc}. 
\item A $J$-twist is an invertible element $\theta \in A$ such that for every object $V\in (A\mod)_i$ the induced map
\begin{align*}
\theta_V: &V\to ^{i}V\\
& v \mapsto \theta^{-1}.v
\end{align*}
is a $J$-twist on $A\mod$ as defined in \ref{ribbon_def}.
\end{enumerate}
If $A$ has an $R$-matrix and a twist, we call it a $J$-\emph{ribbon-algebra}.
\end{Definition}



\begin{Remark}
\begin{itemize}
\item
A $J$-ribbon algebra is not, in general,
a ribbon Hopf algebra.
\item 
The component $A_1$ with the obvious restricitions of $R$ and $\theta$ is a ribbon algebra.
\item The conditions that the category $A\mod$ is braided resp. ribbon can be translated into algebraic coniditions on the elements $R$ and $\theta$. Since we are mainly interested in the categorical structure we refrain from doing that here.
\end{itemize}
\end{Remark}

Taking all the introduced algebraic strucutre into account we get the following proposition.

\begin{Proposition}\label{category-J-ribbon}
The representation category of a $J$-ribbon algebra is a $J$-ribbon category.
\end{Proposition}

\begin{Remark}
In \cite{turaev2010}, Hopf algebras and ribbon Hopf algebras
with strict $J$-action have been considered.
The next subsection will give an illustrative
example where the natural action is not strict.
\end{Remark}

\subsection{Equivariant Drinfel'd Double}

The goal of this subsection is to  construct a $J$-ribbon algebra, given a finite group $G$
with a weak $J$-action. As explained in subsection 
\ref{Def:action}, such a weak $J$-action amounts to
a group extension
\begin{equation}\label{groupext}
1\to G \longrightarrow H \stackrel{\pi}{\longrightarrow} J
\to 1
\end{equation}
with a set-theoretical splitting $s:J \to H$.

We start from the well-known fact reviewed
in subsection \ref{dd} that  the Drinfel'd double 
$\cald(H)$ of the finite group $H$ is a ribbon Hopf 
algebra. The double $\cald(H)$ has a canonical basis
$\delta_{h_1}\otimes h_2$ indexed by pairs
$h_1,h_2$ of elements of $H$. Let $G\subset H$ be
a subgroup. We are interested in
the vector subspace $\cald^J(G)$ spanned by the
basis vectors $\delta_h\otimes g$
with $h\in H$ and $g\in G$.

\begin{Lemma}
The structure maps of the Hopf algebra
$\cald(H)$ restrict to the vector subspace
$\cald^J(G)$ in such a way that the
latter is endowed with the structure of a  Hopf subalgebra.
\end{Lemma}

\begin{Remark}
The induced algebra structure on $\cald^J(G)$ is the
one of the groupoid algebra of the action groupoid
$H//G$.
\end{Remark}

The Drinfel'd double $\cald(H)$ of a group $H$ 
has also the structure of a ribbon algebra. However, 
neither the R-matrix nor the the ribbon element yield
an R-matrix or a ribbon element of $\cald^J(G)\subset\cald(H)$.
Rather, this Hopf subalgebra can be endowed with the 
structure of a $J$-ribbon Hopf algebra as in definition
\ref{def:jribbon}.

To this end, consider the partition of the group $H$
into the subsets  $H_j:=\pi^{-1}(j) \subset H$. It gives
a $J$-grading of the algebra $A$ as a direct sum
of subalgebras:
\[A_j := \langle \delta_h\otimes g\rangle_{h\in H_j,g\in G}
\,\,\, . \]
The set-theoretical section $s$ gives a weak action of $J$
on $A$ that can be described by its action
on the canonical basis of $A_j$:
\begin{eqnarray*}
\varphi_j(\delta_h \otimes g) :=& (\delta_{s(j)hs(j)^{-1}} \otimes s(j)gs(j)^{-1})
\,\, ;
\end{eqnarray*}
the coherence elements are
\[c_{ij} := \sum_{h\in H} \delta_h \otimes 
s(i)s(j)s(ij)^{-1} \,\,\,\, .\]


\begin{Proposition}
The Hopf algebra $\cald^J(G)$, together with the grading
and weak $J$-action derived from the weak $J$-action 
on the group $G$,  has the structure of a $J$-Hopf algebra.
\end{Proposition}
\begin{proof}
It only remains to check the compatibility relations of grading and weak $J$-action
with the Hopf algebra structure that have been formulated in
definition \ref{JHopf}. 
The fact that $\varphi_i(A_j) \subset A_{iji^{-1}}$ is immediate, since $s(i)H_js(i^{-1}) \subset H_{iji^{-1}}$. The axioms follow are essentially equivalent to the property that conjugation commutes with the product and coproduct of the Drinfel'd double.
\end{proof}

We now turn to the last piece of structure, 
an $R$-matrix and twist element in $\cald^J(G)$.
Consider the element 
$R= \sum_{ij}R_{i,j} \in \cald^J(G) \otimes \cald^J(G)$ 
with homogeneous elements $R_{i,j}$ defined as
\begin{eqnarray}\label{JR-matrix}
R_{i,j} := \sum_{h_1 \in H_i, h_2 \in H_{j}} 
(\delta_{h_1} \otimes 1)\otimes(\delta_{h_2} \otimes s(i^{-1}) h_1).
\end{eqnarray}
(Note that $\pi(s(i^{-1})h_1) = 1$ for $h_1 \in H_i$ and thus $s(i^{-1})h_1 \in G$.)\\
The element $R_{i,j}$ is invertible with inverse
\begin{eqnarray*}
R_{i,j}^{-1} = \sum_{h_1 \in H_i, h_2 \in H_{j}} 
(\delta_{h_1} \otimes 1)\otimes(\delta_{h_2} 
\otimes h_1^{-1} s(i^{-1})^{-1}).
\end{eqnarray*}
We also introduce a twist element
$\theta = \sum_{j \in J} \theta_j \in \cald^J(G)$
with 
\begin{align}\label{J-twist}
\theta_j^{-1} := \sum_{h\in H_j} 
\delta_h \otimes s\big(j^{-1}\big)h \in A_j
\end{align}
for every 
element $j \in J$.

\begin{Proposition}\label{double-J-ribbon}
The elements $R$ and $\theta$ endow the $J$-Hopf algebra
$\cald^J(G)$  with the structure of a $J$-ribbon algebra
that we call the $J$-Drinfel'd double of $G$.
\end{Proposition}

\begin{proof}
By Definition \ref{def:jribbon} we have to check that the induced tranformations on the level of categories satisify the axioms for a braiding and a twist. We first compute the induced braiding by using the R-matrix given in \eqref{JR-matrix}:
For $V\in (\cald^J(G)\mod)_j = \cald^J(G)_j\mod$ and $W\in \cald^J(G)\mod$, we get the linear map
\begin{align*}
c_{V,W}:V\otimes W &\to ^jW \otimes V\\
v\otimes w & \mapsto s(j^{-1})h.w\otimes v \quad\text{for} \quad v \in V_h
\end{align*}
First of all, we show that this is a morphism in $\cald^J(G)\mod$.
Let $v\in V_h$ and $w \in W_{h'}$. We have $v\otimes w \in (V\otimes W)_{hh'}$ and
$c_{V,W}(v\otimes w) = s(j^{-1})h.w\otimes v \in (^jW\otimes V)_{hh'}$, since the element $s(j^{-1})h.w$ is in the component $\left(s(j^{-1})hh'h^{-1}s(j^{-1})^{-1}\right)$ of $W$ which is the component $hh'h^{-1}$ of $^jW$. So $c_{V,W}$ respects the grading.
As for the action of $G$, we observe that for $g \in G$, the element $g.v$ lies in the component $V_{ghg^{-1}}$, and so
\begin{align*}
c_{V,W}(g.(v\otimes w)) = \left(s(j^{-1})ghg^{-1}\right)g.w\otimes g.v =  s(j^{-1})gh.w\otimes g.v\\
g.c_{V,W}(v\otimes w) = \left(s(j^{-1})gs(j^{-1})^{-1}\right)s(j^{-1})h.w \otimes g.v =  s(j^{-1})gh.w\otimes g.v
\end{align*}
which shows that $c_{V,W}$ commutes with the action of $G$.\\
Furthermore, it needs to be checked, that $c_{V,W}$ satisfies the hexagon axioms, which in our case reduce to the two diagrams
\begin{align*}
\xymatrix{
U\otimes V \otimes W\ar[rr]^{c_{U, V\otimes W}}\ar[d]_{c_{U,V}\otimes \id_W}  && ^{i}V\otimes ^{i}W\otimes U\\
^{i}V \otimes U \otimes W\ar[rru]_{\id_{^{i}V} \otimes c_{U,W}}&&
}
\end{align*}
and
\begin{align*}
\xymatrix{
U\otimes V \otimes W\ar[rr]^{c_{U\otimes V,W}}\ar[d]_{\id_U \otimes c_{V,W}}  &&^{ij}W \otimes U \otimes V\\
U \otimes ^{j}W  \otimes V\ar[rr]_{c_{U,^{j}W} \otimes \id_{V}} && ^{i}(^{j}W) \otimes U \otimes V
\ar[u]_{\alpha_{i,j}\otimes \id_U}} 
\end{align*}
where  $U \in \calc_i, V \in \calc_j$ and we suppressed the associators in $A\mod$.\\
And at last we need to check the compatibility of the braiding and the $J$-action, i.e. diagram \eqref{action-braiding-compatibility}. 
All of that can be proven by straightforward calculations, which show, that the morphism induced by the element $R$ is really a braiding in the module category.\\
We now compute the morphism given by the action with the inverse twist element \eqref{J-twist}. For $V\in (\cald^J(G)\mod)_j = \cald^J(G)_j\mod$ we get the linear map:
\begin{align*}
\theta_V:V &\to ^{j}V\\
v&\mapsto s(j^{-1})h.v
\end{align*}
for $v \in V_h$. This map is compatible with the $H$-grading since for $v \in V_h$, we have $s(j^{-1})h.v \in V_{s(j^{-1})hs(j^{-1})^{-1}} = (^j V)_h$ and it is also compatible with the $G$-action, since
\begin{align*}
\theta_V(g.v) &= s(j^{-1})(ghg^{-1})g.v = s(j^{-1})gh.v\\
g.\theta_V(v) &= \left(s(j^{-1})gs(j^{-1})^{-1}\right)s(j^{-1})h.v = s(j^{-1})gh.v
\end{align*}
So the induced map is a morphism in the category $\cald^J(G)$. In order to proof that it is a twist, we have to check the commutativity of diagram \eqref{com.diagram-twist} as well as for every $i,j \in J$ and object $V$ in the component $j$, the commutativity of the diagrams
\begin{align*}
\xymatrix{
^{j}V^{\vee}\ar[r]^{(\theta_V)^{\vee}}\ar[d]_{\theta_{^{j}V^{\vee}}}  &V^{\vee}\\
^{j^{-1}}(^{j}V^{\vee})\ar[ru]_{\alpha_{j^{-1},j}}&} 
\end{align*}
(which displays the compatibility of the twist with the duality),
and
\begin{align*}
\xymatrix{
^{i}V\ar[rr]^{ ^{i}\theta_V}\ar[d]_{\theta_{^{i}V}}  && ^{i}(^{j}V)\ar[d]^{\alpha_{i,j}}\\
^{iji^{-1}}(^{i}V) \ar[rr]_{\alpha_{iji^{-1},i}}&& ^{ij}V} 
\end{align*}
(which is the compatibility of the twist with the $J$-action). This again can be checked by straightforward calculations.
\end{proof}

We are now ready to come back to the $J$-equivariant
tensor category $\cgj$ described in proposition 
\ref{prop:Jtensor}.
From this proposition, we know that the category
$\cdw$ is equivalent to $H//G\mod\cong\cald^J(G)\mod$ as a $J$-equivariant
tensor category. 
Also $J$-action and tensor product coincide with
the ones on $\cald^J(G)\mod$. 
Moreover, the equivariant braiding of $\cdw$ computed in proposition 
\ref{prop:braiding}  coincides with the braiding on $\cald^J(G)$ computed in the proof of the last proposition 
\ref{double-J-ribbon}.

This allows us to
transfer also the other structure on the representation category
of the $J$-Drinfel'd double $\cald^J(G)$ described
in proposition \ref{double-J-ribbon} to the category
$\cdw$:

\begin{Proposition}\label{J-fusion}
The $J$-equivariant tensor category $\cgj$ described
in proposition \ref{prop:Jtensor} can be endowed with
the structure of a $J$-premodular category
such that it is equivalent, as a $J$-premodular
category, to the category $\cald^J(G)\mod$.
\end{Proposition}

\begin{Remark}
At this point, we have constructed in paticular a $J$-equivariant 
braided fusion category $\cgj$ (see \cite{ENO2009}) with neutral component
$\calc(G)_1\cong\cald(G)\mod$
from a weak action of the group $J$ on the group $G$,
or in different words, from a 2-group homomorphisms 
$J \to \AUT(G)$ with $\AUT(G)$ the automorphism 2-group 
of $G$. 

In this remark, we very briefly sketch the relation
to the description of $J$-equivariant braided fusion categories
with given neutral sector $\calb$ in terms of 3-group 
homomorphisms $J \to \BrPic(\calb)$ given in \cite{ENO2009}.
Here $\BrPic(\calb)$ denotes the so called Picard 3-group
whose objects are invertible module-categories of the 
category $\calb$. The group structure comes from 
the tensor product of module categories which can be 
defined since the braiding on $\calb$ allows to turn module categories 
into bimodule categories.

Using this setting, we give a description of our 
$J$-equivariant braided fusion category $\cald^J(G)\mod$
in terms of a functor $\Xi: J \to \BrPic(\cald(G))$.
To this end, we construct a 3-group homomorphism 
$\AUT(G) \to \BrPic(\cald(G))$ and write $\Xi$ as the 
composition of this functor and the functor 
$J \to \AUT(G)$ defining the weak $J$-action. 

The 3-group homomorphism $\AUT(G) \to \BrPic(\cald(G))$ 
is given as follows: to an object $\varphi \in \AUT(G)$ 
we associate the twisted conjugation groupoid 
$G//^{\varphi}G$, where $G$ acts on itself by 
twisted conjugation, $g.x := g x \varphi(g)^{-1}$. 
This yields the category 
$G//^{\varphi}G\mod := [ G//^{\varphi}G , \vect_\KK ]$
which is naturally a module category over $\cald(G) \mod$. 
Morphisms $\varphi \to \psi$ in $\AUT(G)$ 
are given by group elements $g \in G$ with 
$g \varphi g^{-1} = \psi$; to such a morphism we associate 
the functor $L_g: G//^{\varphi}G \to G//^{\psi}G$ given by 
conjugating with $g\in\ G$ on objects and morphisms. 
This induces functors of module categories
$G//^{\varphi}G\mod \to G//^{\psi}G\mod$. Natural
coherence data exist; one then shows 
that this really establishes the desired 3-group homomorphism.
\end{Remark}

\subsection{Orbifold category and orbifold algebra}

It remains to show that the $J$-equivariant ribbon
category $\cgj$ described in \ref{double-J-ribbon}
is $J$-modular. To this end, we will use
the orbifold category of the $J$-equivariant category:

\begin{Definition}
Let $\calc$ be a $J$-equivariant category. The 
\emph{orbifold category} $\calc^J$ of $\calc$ has:
\begin{itemize}
\item as objects pairs $(V,(\psi_j)_{j \in J})$
consisting of an object $V\in\calc$ and a family
of isomorphisms $\psi_j:{}^jV \to V$ with $j\in J$
such that $\psi_i\,\circ\, {}^i\psi_j = \psi_{ij}\circ \alpha_{i,j}$.
\item as morphisms $f:(V,\psi^V_j) \to (W,\psi^W_j)$
those morphisms $f:V \to W$ in $\calc$ for which
$\psi_j \circ {}^j(f) = f \circ \psi_j$ holds for 
all $j \in J$.
\end{itemize}
\end{Definition}

In \cite{kirI17}, it has been shown that the orbifold 
category of a $J$-ribbon category is an ordinary,
non-equivariant ribbon category:

\begin{Proposition}\label{kirillov1}
\begin{enumerate}
\item
Let $\calc$ be a $J$-ribbon category. Then the orbifold 
category $\calc^J$ is naturally endowed with the
structure of a ribbon category by the following data:
\begin{itemize}
\item The tensor product of the objects $(V,(\psi^V_j))$
and $(W,(\psi^W_j))$ is defined as the object 
$(V\otimes W, (\psi^V_j\otimes \psi^W_j))$.

\item The tensor unit for this tensor product
is $\mathbf{1} = (\mathbf{1},(\id))$

\item The dual object of $(V,(\psi_j))$ is 
the object $(V^{\ast},(\psi_j^{\ast})^{-1})$, 
where $V^{\ast}$ denotes the dual object in $\calc$. 

\item The braiding of the two objects $(V,(\psi^V_j))$ and
$(W,(\psi^W_j))$ with $V \in \calc_j$ is given by the
isomorphism
$(\psi_j \otimes \id_V)\circ c_{V,W}$, where 
$c_{V,W}:V \otimes W \to ^jW \otimes V$ is the $J$-braiding 
in $\calc$.

\item The twist on an object $(V,(\psi_j))$ is 
$\psi_j \circ \theta$, where $\theta:V \to {}^jV$ is the twist in $\calc$.
\end{itemize}
\item
If $\calc$ is a $J$-premodular category, then the 
orbifold category $\calc^J$ is even a premodular category.
\end{enumerate}
\end{Proposition}

It has been shown in \cite{kirI17} that the $J$-modularity 
of a $J$-premodular category is equivalent to 
the modularity as in definition \ref{modularity} of its 
orbifold category. Our problem is thus reduced to showing
modularity of the orbifold category of $\cald^J(G)\mod$.

To this end, we describe orbifoldization on
the level of (Hopf-)algebras: given a $J$-equivariant
algebra $A$, we introduce an orbifold algebra $\widehat{A}^J$
such that its representation category $\widehat{A}^J\mod$
is isomorphic to the orbifold category of $A\mod$.

\begin
{Definition}\label{Def:o.algebra}
Let $A$ be an algebra with a weak $J$-action
$(\varphi_j,c_{ij})$.
We endow the  vector space 
$\widehat{A}^J:=A \otimes \KK[J]$
with a unital associative multiplication which is defined on an element
of the form $(a\otimes j)$ with $a \in A$ and $j \in J$
by\[(a\otimes i)(b\otimes j) := a \varphi_i(b)c_{ij}\otimes ij
\,\, . \]

This algebra  is called the 
\emph{orbifold algebra} $\widehat{A}^J$ of the
$J$-equivariant
algebra $A$ with respect to the weak $J$-action.
\end{Definition}

If $A$ is even a $J$-Hopf algebra, it is possible
to endow the orbifold algebra with even more structure.
To define the coalgebra structure on the orbifold algebra,
we use the standard coalgebra structure on the
group algebra $\KK[J]$ with coproduct 
$\Delta_J(j) = j\otimes j$ and counit $\epsilon_J(j) = 1$ 
on the canonical basis $(j)_{j\in J}$. The
tensor product coalgebra on $A\otimes \KK[J]$ has
the coproduct and counit
\begin{equation} 
\Delta(a\otimes j) = (\id_A \otimes \tau \otimes 
\id_{\KK[J]})(\Delta_A(a) \otimes j \otimes j ),\quad\text{ and }\quad
\epsilon(a\otimes j) = \epsilon_A(a) 
\end{equation}
which is clearly coassociative and counital.

To show that this endows the orbifold algebra with the
structure of a bialgebra, we have first to show that 
the coproduct $\Delta$ is a unital algebra morphism.
This follows from the fact, that $\Delta_A$ is already 
an algebra morphism and that the action of $J$ is by 
coalgebra morphisms. Next, we have to show that
the counit $\epsilon$ is a unital algebra morphism as well.
This follows from the fact that the action of $J$ 
commutes with the counit and from the fact that the elements $c_{i,j}$ are group-like.
The compatibility of $\epsilon$ with the unit is obvious.

In a final step, one verifies that the endomorphism
\[S(a\otimes j) = (c_{j^{-1}, j})^{-1}\varphi_{j^{-1}}(S_A(a))\otimes j^{-1}\]
is an antipode. Altogether, one arrives at

\begin{Proposition}
If $A$ is a $J$-Hopf-algebra, then the 
orbifold algebra $\widehat{A}^J$ has a natural
structure of a Hopf algebra. 
\end{Proposition}

\begin{Remark}
\begin{enumerate}
\item 
The algebra $\widehat A^J$ is not the fixed point 
subalgebra $A^J$ of $A$; in general, the categories
$A^J\mod$ and $\widehat A^J\mod$ are inequivalent.

\item
Given any Hopf algebra $A$ with weak $J$-action,
we have an exact sequence of Hopf algebras
\begin{align}\label{Hopf-algebra-sequ.}
A \to \widehat{A}^J \to \KK[J] \,\,\, .
\end{align}
In particular, $A$ is a sub-Hopf algebra of $\widehat{A}^J$. In general, there is no inclusion of $\KK[J]$ into $\widehat{A}^J$
as a Hopf algebra.
 
\item 
If the action of $J$ on the algebra $A$ is strict, then 
the algebra $A$ is a module algebra over the Hopf algebra 
$\KK[J]$ (i.e.\ an algebra in the tensor category 
$\KK[J]\mod$). Then the orbifold algebra is the smash 
product $A\# \KK[J]$ 
(see \cite[Section 4]{montgomery1993hopf} 
for the definitions). The situation described occurs,
if and only if the exact sequence (\ref{Hopf-algebra-sequ.}) 
splits.
\end{enumerate}
\end{Remark}

The next proposition justifies the name ``orbifold algebra''
for $\widehat{A}^J$:

\begin{Proposition}\label{orbifoldization}
Let $A$ be a $J$-Hopf algebra. Then there is an equivalence
of tensor categories
\[\widehat{A}^J\mod \cong (A\mod)^J \,\, . \]
\end{Proposition}

\begin{proof}
\begin{itemize}
\item
An object of $(A\mod)^J$ consists of a $\KK$-vector
space $M$, an $A$-action $\rho: A\to \End(M)$
and a family of $A$-module morphisms $(\psi_j)_{j\in J}$.
We define on the same $\KK$-vector space $M$ the
structure of an $\widehat{A}^J$ module
by $\tilde\rho: \widehat{A}^J \to\End(M)$ with
$\tilde\rho(a\otimes j) := \rho(a)\circ(\psi_{j^{-1}})^{-1}$.


One next checks that, given two objects
$(M,\rho,\psi)$ and $(M',\rho',\psi')$
in $(A\mod)^J$, a $\KK$-linear map 
$f\in \Hom_\KK(M,M')$ is in the subspace
$\Hom_{(A\mod)^J}(M,M')$ if and only if it is
in the subspace 
$\Hom_{\widehat{A}^J\mod}(M,\tilde\rho),(M',\tilde\rho'))$.

We can thus consider a $\KK$-linear functor
\begin{equation}\label{orbifold-equivalence}
F: (A\mod)^J \to \widehat{A}^J\mod
\end{equation}
which maps on objects by $(M,\rho,\psi)\mapsto (M,\tilde
\rho)$ and on morphisms as the identity. This functor
is clearly fully faithful. 

To show that the functor is also essentially surjective,
we note that for any object $(M,\tilde\rho)$ in 
$\widehat{A}^J\mod$, an object in $(A\mod)^J$ can
be obtained as follows: on the underlying vector space,
we have the structure of an $A$-module by restriction,
$\rho(a):=\tilde\rho(a\otimes 1_J)$. A family of
equivariant morphisms is given by
$\psi_j :=(\tilde\rho(\mathbf{1}\otimes j^{-1}))^{-1}$. Clearly
its image under $F$ is isomorphic to $(M,\tilde\rho)$. 
This shows that the functor $F$ is an equivalence of
categories, indeed even an isomorphism of categories.

\item The functor $F$ is also a strict tensor functor:
consider two objects  $F(M,\rho,\psi)$ and $F(M',\rho',\psi')$
in $(A\mod)^J$. The functor
$F$ yields the
following action of the orbifold Hopf algebra
$\widehat{A}^J$ on the $\KK$-vector space $M\otimes_\KK M'$:
$$ \tilde\rho_{M\otimes M'}(a\otimes j)
= \rho\otimes \rho'(\Delta(a)) \circ ((\psi_{j^{-1}})^{-1}
\otimes (\psi'_{j^{-1}})^{-1})\,\, .$$
Since the coproduct on $\widehat{A}^J$ was just given
by the tensor product of coproducts on
$A$ and $\KK[J]$, this coincides with the tensor product 
of $F(M,\rho,\psi)$ and $F(M',\rho',\psi')$
in $\widehat{A}^J\mod$.
\end{itemize}
\end{proof}

In a final step, we assume that the $J$-equivariant algebra $A$
has the additional structure of a $J$-ribbon algebra.
Then, by proposition \ref{category-J-ribbon},
the category $A\mod$ is a $J$-ribbon category
and by proposition \ref{kirillov1} the orbifold
category $(A\mod)^J$ is a ribbon category.
The strict isomorphism (\ref{orbifold-equivalence})
of tensor categories allows us to transport both the
braiding and the ribbon structure to the representation category
of the orbifold Hopf algebra $\widehat{A}^J$.
General results \cite[Proposition 16.6.2]{kassel1995quantum}
assert that this amounts to a natural structure of
a ribbon algebra on $\widehat{A}^J$. In fact, we directly
read off the R-matrix and the ribbon element. For example,
the R-matrix $\hat R$ of $\hat A^J$ equals
$$\ \hat R= \hat\tau \,\,
c_{\widehat{A}^J,\widehat{A}^J}
(1_{\widehat{A}^J}\otimes 1_{\widehat{A}^J})
\in \widehat{A}^J\otimes \widehat{A}^J \,\,\, , $$
where the linear map $\hat\tau$
flips the two components of the tensor product
$\widehat{A}^J\otimes \widehat{A}^J$. This expression can be
explicitly evaluated, using the fact that 
$A \otimes \KK[J]$ is an object in $(A\mod)^J$ with 
$A$-module structure given by left action
on the first component and that the morphisms $\psi_j$ 
are given by left multiplication on the second component. 
We find for the  R-matrix of $\widehat{A}^J$
\begin{align*}
\hat R & = \sum_{i,j \in J}(\id \otimes \psi_j)(\rho\otimes \rho)(R_{ij})
((1_A \otimes 1_J) \otimes 1_A \otimes 1_J)\\
& = \sum_{i,j \in J}((R_{i,j})_{1}\otimes 1_J) 
\otimes (1_A \otimes j^{-1})^{-1}((R_{i,j})_{2} \otimes 1_J))
\end{align*}
where $R$ is the R-matrix of $A$.
The twist element of $\widehat{A}^J$ can be computed 
similarly; one finds
\[\theta^{-1} = \sum_{j\in J}\psi_j\circ 
\rho(\theta_j)(1_{A_j} \otimes 1_J)
=\sum_{j\in J}(1_A \otimes j^{-1})^{-1} (\theta_j \otimes 1_J)\]
We summarize our findings:

\begin{corollary}\label{ribbon-ribbon}
If $A$ is a $J$-ribbon algebra, then the orbifold algebra 
$\widehat{A}^J$ inherits a natural structure of a ribbon 
algebra such that the equivalence of tensor categories in
proposition \ref{orbifoldization} is an equivalence of 
ribbon categories.
\end{corollary}

\subsection{Equivariant modular categories}

In this subsection, we show that the orbifold category
of the $J$-equivariant ribbon
category $\cdw\mod$ is $J$-modular. A theorem of
Kirillov \cite[Theorem 10.5]{kirI17} then immediately implies
that the category $\cdw\mod$ is $J$-modular.

Since we have already seen in corollary \ref{ribbon-ribbon}
that the orbifold category is equivalent, as a
ribbon category, to the representation category of the
orbifold Hopf algebra, it suffices the compute
this Hopf algebra explicitly. Our final result asserts
that this Hopf algebra is an ordinary 
Drinfel'd double:

\begin{Proposition}\label{J-Drinfeld-orbifold}
The $\KK$-linear map
\begin{equation}
\begin{array}{rll}
\Psi:\qquad
\widehat{\cald^J(G)}^J &\to & \cald(H)\\[.3em]
(\delta_h \otimes g \otimes j) & \mapsto &(\delta_h \otimes gs(j))
\end{array}
\end{equation}
is an isomorphism of ribbon algebras, where the Drinfel'd
double $\cald(H)$ is taken with the standard ribbon 
structure introduced in subsection 2.5.
\end{Proposition}

This result immediately implies the equivalence
\[(\widehat{\cald^J(G)}\mod)^J \cong \cald(H)\mod\]
of ribbon categories and thus, by proposition \ref{modular-Drinfeld},
the modularity of the orbifold category, so that we have
finally proven:

\begin{Theorem}\label{J-modular}
The category $\cgj$ has a natural structure of a 
$J$-modular tensor category.
\end{Theorem}

\begin{proof}[Proof of proposition \ref{J-Drinfeld-orbifold}]
We show by direct computations that the linear map 
$\Psi$ preserves product, coproduct, R-matrix and 
twist element:
\begin{itemize}
\item
Compatibility with the product:
\begin{align*}
\Psi((\delta_h \otimes g \otimes j)&(\delta_h' \otimes g' \otimes j')) 
= \Psi((\delta_h \otimes g)\cdot {}^j(\delta_h' \otimes g')c_{jj'} \otimes jj') \\[.2em] 
& = \Psi\left((\delta_h \otimes g)\cdot (\delta_{s(j)h's(j)^{-1}} \otimes s(j)g's(j)^{-1})\cdot\sum_{h''\in H} (\delta_{h''} \otimes s(j)s(j')s(jj')^{-1})\otimes jj'\right) \\[.2em] 
& = \Psi\left(\delta(h,gs(j)hs(j)^{-1}g^{-1})(\delta_h \otimes gs(j)g's(j')s(jj')^{-1})\otimes jj'\right) \\[.2em] 
& = \delta(h,gs(j)hs(j)^{-1}g^{-1})(\delta_h \otimes gs(j)g's(j')) \\[.2em] 
& = (\delta_{h}\otimes gs(j))\cdot (\delta_{h'}\otimes g's(j'))\\[.2em] 
& = \Psi(\delta_h \otimes g \otimes j)\Psi(\delta_{h'} \otimes g' \otimes j') 
\end{align*}
\item 
Compatibility with the coproduct:
\begin{align*}
(\Psi \otimes \Psi)\Delta(\delta_h \otimes g \otimes j) 
&= \sum_{h'h'' = h} \Psi(\delta_{h'} \otimes g \otimes j) \otimes \Psi(\delta_{h''} \otimes g \otimes j)\\[.2em]
&= \sum_{h'h'' = h} (\delta_{h'} \otimes gs(j)) \otimes (\delta_{h''} \otimes gs(j))\\[.2em]
&= \Delta\left(\Psi(\delta_h \otimes g \otimes j)\right) 
\end{align*}
\item 
The R-matrix of the orbifold algebra $\widehat{\cald^J(G)}^J $ can
be determined using the lines preceding 
corollary \ref{ribbon-ribbon} and the definition of the R-Matrix of $\cald^J(G)$ given in \eqref{JR-matrix}:
$$R = \sum_{j,j'\in J}\sum_{h \in H_j, h' \in H_{j'}} (\delta_{h} 
\otimes 1_G \otimes 1_J) \otimes (1 \otimes 1_G \otimes j^{-1})^{-1}(\delta_{h'} \otimes s(j^{-1})h
\otimes 1_J) $$
This implies
\begin{align*}
(\Psi \otimes \Psi)(R) 
&= \sum_{j,j'\in J}\sum_{h \in H_j, h' \in H_{j'}} \Psi(\delta_{h} \otimes 1_G \otimes 1_J) \otimes \Psi(1 \otimes 1_G \otimes j^{-1})^{-1} \cdot \Psi(\delta_{h'} \otimes s(j^{-1})h \otimes 1_J)\\[.2em]
&=\sum_{j,j'\in J}\sum_{h \in H_j, h' \in H_{j'}} (\delta_{h} \otimes 1_H) \otimes (1 \otimes s(j^{-1})^{-1}) \cdot (\delta_{h'} \otimes s(j^{-1})h)\\[.2em]
&=\sum_{h \ h' \in H}(\delta_{h} \otimes 1) \otimes (\delta_{h'} \otimes h),\\[.2em]
\end{align*}
which is the standard R-matrix of the Drinfel'd double $\cald(H)$.
\item The twist in $\widehat{\cald^J(G)}^J $ is by corollary \ref{ribbon-ribbon} equal to
\begin{equation*}
\theta^{-1} = \sum_{j\in J}\sum_{h\in H_j} (\delta_h \otimes s(j^{-1})h\otimes 1_J)
\end{equation*}
and thus it gets mapped to the element
\begin{align*}
\Psi(\theta^{-1})
&= \sum_{j\in J}\sum_{h\in H_j} \Psi(1 \otimes 1_G \otimes j^{-1})^{-1} \Psi(\delta_h \otimes s(j^{-1})h\otimes 1_J)\\[.2em]
&= \sum_{j\in J}\sum_{h\in H_j} (1 \otimes s(j^{-1})^{-1}) \cdot (\delta_h \otimes s(j^{-1})h)\\[.2em]
&=\sum_{h\in H} (\delta_h \otimes h)
\end{align*}
which is the inverse of the twist element in $\cald(H)$.
\end{itemize}
\end{proof}

\subsection{Summary of all tensor categories involved}

We summarize our findings by discussing again the
four tensor categories mentioned in the introduction,
in the square of equation (\ref{int:square}),
thereby presenting the
explicit solution of the algebraic problem described in section
\ref{intro:1.1}. 
Given a finite group $G$
with a weak action of a finite group $J$, we get an
extension $1\to G \to H \to J\to 1$ of finite groups,
together with a set-theoretic section $s:J\to H$.

\begin{Proposition}\mbox{} \\
We have the following natural realizations of the categories
in question in terms of categories of finite-dimensional representations
over finite-dimensional ribbon algebras:
\begin{enumerate}
\item The premodular category introduced in \cite{bantay} is 
$\calb(G\triangleleft H)\mod$. As an abelian category, it is equivalent
to the representation category $G//H\mod$ of the action groupoid $G//H$, i.e.\
to the category of $G$-graded $\KK$-vector spaces with compatible
action of $H$.

\item The modular category obtained by modularization 
is $\cald(G)\mod$. As an abelian category, it is equivalent to
$G//G\mod$.

\item
The $J$-modular category constructed in this paper is
$\cald^J(G)\mod$. As an abelian category, it is equivalent to
$H//G\mod$.

\item The modular category obtained by orbifoldization 
from the $J$-modular category $\cald(G)\mod$ is equivalent to
$\cald(H)\mod$. As an abelian category, it is equivalent to
$H//H\mod$.

\end{enumerate}
Equivalently, the diagram in equation (\ref{int:square}), has
the explicit realization:
\begin{equation}\label{square49}
\xymatrix{
\cald(G)\mod \ar@<0.7ex>[d]^{\text{orbifold}} \ar@(lu,ur)[]^J\ar@{^{(}->}[r] 
& \cald^J(G)\mod \ar@(lu,ur)[]^J \ar@<0.7ex>[d]^{\text{orbifold}} \\
\calb(G\triangleleft H)\mod \ar[u]^{\text{modularization}}\ar@{^{(}->}[r] 
& \cald(H)\mod \ar[u]
}
\end{equation}

\end{Proposition}

We could have chosen the inclusion in the lower line as an
alternative starting point for the solution of the algebraic 
problem presented in introduction \ref{intro:1.1}. Recall
from the introduction that the category $\calb(G\triangleleft H)\mod$
contains a Tannakian subcategory that can be identified with the
category of representations of the quotient group $J=H/G$.
The Tannakian subcategory and thus the category $\calb(G\triangleleft H)\mod$
contain a commutative Frobenius algebra given by the algebra of functions
on $J$; recall that the modularization function was just induction along 
this algebra. The image of this
algebra under the inclusion in the lower line yields
a commutative Frobenius algebra in the category $\cald(H)\mod$. 
In a next step, one can consider induction along this algebra to
obtain another tensor category which, by general results 
\cite[Theorem 4.2]{kirI17} is a $J$-modular category.

In this approach, it remains to show that this $J$-modular
tensor category is equivalent, as a $J$-modular tensor category, 
to $\cald^J(G)\mod$ and, in a next step that the modularization
$\cald(G)\mod$ can be naturally identified with the neutral sector
of the $J$-modular category. This line of thought has been discussed in
\cite[Lemma 2.2]{kirI15} including the square \eqref{square49} of Hopf algebras. Our results directly lead to a natural
Hopf algebra $\cald^J(G)$ and additionally show how
the various categories arise from extended topological field theories
which are built on clear geometric principles and through which
all additional structure of the algebraic categories become explicitly
computable.

\section{Outlook}

Our results very explicitly provide an interesting class
$J$-modular tensor categories. All data of these theories,
including the representations of
the modular group $SL(2,\Z)$ on the vector spaces assigned to the
torus, are directly accessible in terms of representations of
finite groups. Also series of examples exist in which closed
formulae for all quantities can be derived, e.g.\ for 
the inclusion of the alternating group in the symmetric group.

Our results admit generalizations in various directions.
In fact, in this paper, we have only studied a
subclass of Dijkgraaf-Witten theories. The general
case requires, apart from the choice of a finite group $G$,
the choice of an element of
$$\ H^3_{Gp}(G,U(1))= H^4(\cala_G,\Z)\, . $$
This element can be interpreted \cite{Wi} geometrically 
as a 2-gerbe on $\cala_G$. It is known that in this
case a quasi-triangular Hopf algebra can be extracted that
is exactly the one discussed in \cite{dipr}. Indeed,
our results can also be generalized by including 
the additional choice of a non-trivial element 
$$\omega\in H^4_J(\cala_G,\Z)\equiv H^4(\cala_G//J,\Z)
\,\, . $$
Only all these data together
allow to investigate in a similar manner the
categories constructed by Bantay \cite{bantay} 
for crossed modules with a boundary map that is not necessarily
injective any longer. We plan to explain this general case
in a subsequent publication.

\renewcommand{\theDef}{\Alph{section}.\arabic{subsection}.\arabic{Def}}

\renewenvironment{Definition}{ \refstepcounter{Def}\mbox{}\\
\noindent\sl\textbf{Definition
\Alph{section}.\arabic{subsection}.\arabic{Def}} }
{\vspace{0.3cm}}

\appendix
\section{Appendix}

\subsection{Cohomological description of twisted bundles}\label{cech}

In this appendix, we give a description of $P$-twisted 
bundles as introduced in definition \ref{def:twisted} in terms of
local data. This local description
will also serve as a motivation for the term `twisted' in
twisted bundles. 
Recall the relevant situation: 
$1\to G \to H \stackrel\pi\to J \to 1$ is an exact sequence of groups.
Let $P\stackrel J\to M$ be a $J$-cover.
A $P$-twisted bundle on a smooth manifold $M$ is an $H$-bundle 
$Q\to M$, together with a smooth map $\varphi:Q\to P$ such
that $\varphi(qh)=\varphi(q)\pi(h)$ for all $q\in Q$ and
$h\in H$. \\

We start with the choice of a contractible open covering 
$\{U_\alpha\}$ of $M$, i.e. a covering for which all 
open sets $U_{\alpha}$ are contractible. Then the
$J$-cover $P$ admits local sections over $U_\alpha$. 
By choosing local sections $s_\alpha$, we obtain the cocycle 
$$j_{\alpha \beta} := s_\alpha^{-1} \cdot s_\beta: 
U_\alpha \cap U_\beta \to J$$
describing $P$.  

Let $(Q, \varphi)$ be a $P$-twisted $G$-bundle over $M$. 
We claim that we can find local sections 
$$t_\alpha: U_\alpha \to Q$$
of the $H$-bundle $Q$ which are compatible with the local
section of the $J$-cover $P$ in the sense that  
$\varphi \circ t_\alpha = s_\alpha$ holds for all 
$\alpha$. 

To see this, consider the map $\varphi: Q \to P$; 
restricting the $H$-action on $Q$ along the inclusion
$G \to H$, we get a $G$-action on $Q$ that covers the
identity on $P$. Hence $Q$ has the structure of a
$G$-bundle over $P$. Note that the image of
$s_\alpha$ is contractible, since $U_\alpha$ is
contractible. Thus the $G$-bundle $Q \to P$ admits a 
section $s'_\alpha$ over the image of $s_\alpha$.
Then $t_\alpha := s'_\alpha \circ s_\alpha$ is
a section of the $H$-bundle $Q\to M$ that does the job.

With these sections $t_\alpha: U_\alpha \to Q$, we obtain
the cocycle description
$$ h_{\alpha\beta} := t_\alpha^{-1} \cdot t_\beta :
U_\alpha \cap U_\beta \to H $$
of $Q$. 

The set underlying the group $H$ is isomorphic to the
set  $G \times J$. The relevant multiplication on this
set depends on the choice of a section $J\to H$;
it has been described in equation (\ref{mult}):
$$
(g,i) \cdot (g',j) := \big(g \cdot ~^i (g') \cdot 
c_{i,j}~, ~ij\big)\,\,. 
$$
This allows us to express the $H$-valued cocycles
$h_{\alpha\beta}$ in terms of $J$-valued and
$G$-valued functions 
$$ g_{\alpha\beta} : U_\alpha \cap U_\beta \to G \,\, .$$
By the condition $\varphi \circ t_\alpha = s_\alpha$,
the $J$-valued functions are determined to be 
the $J$-valued cocycles $j_{\alpha\beta}$. Using the multiplication
on the set $G\times J$, the cocycle condition
$h_{\alpha\beta} \cdot h_{\beta \gamma} = h_{\alpha \gamma}$ 
can be translated into the following condition for 
$g_{\alpha\beta}$
\begin{equation}\label{cocycle}
g_{\alpha\beta} \cdot ~^{j_{\beta\gamma}}\big(g_{\beta\gamma}\big) \cdot c_{j_{\alpha\beta},j_{\beta\gamma}} = g_{\alpha\gamma}
\end{equation}
over $U_\alpha \cap U_\beta \cap U_\gamma$. This local 
expression can serve as a justification of the term $P$-twisted
$G$-bundle.
\\

We next turn to morphisms.
A morphism $f$ between $P$-twisted bundles $(Q,\varphi)$ 
and $(Q',\psi)$ which are represented by twisted
cocycles $g_{\alpha\beta}$ and $g'_{\alpha\beta}$ 
is represented by a coboundary 
$$ l_\alpha := (t'_\alpha)^{-1} \cdot f(t_\alpha): U_\alpha \to H $$
between the $H$-valued cocycles $h_{\alpha\beta}$ and 
$h'_{\alpha\beta}$. Since $f$ satisfies $\psi \circ f = \varphi$,
the $J$-component $\pi\circ l_\alpha: U_\alpha \to H \to J$ is given by the 
constant function to $e \in J$. Hence the local
data describing the morphism $f$ reduce to a family of 
functions
$$ k_\alpha : U_\alpha \to G. $$
Under the multiplication (\ref{mult}), the coboundary 
relation $l_\alpha \cdot h_{\alpha\beta} 
= h'_{\alpha\beta} \cdot l_\beta$ translates into
$$
k_\alpha \cdot ~^e\big(g_{\alpha\beta}\big) \cdot 
c_{e,j_{\alpha\beta}} = 
g'_{\alpha\beta} \cdot ~^{j_{\alpha\beta}}\big(k_\beta\big) \cdot c_{j_{\alpha\beta},e}
$$
One can easily conclude from the definition \ref{Def:action} 
of a weak action that $~^e\!g = g$ and $c_{e,g} = c_{g,e} = e$ 
for all $g \in G$. Hence this condition reduces to the condition
\begin{equation}\label{coboundary}
k_\alpha \cdot g_{\alpha\beta} = 
g'_{\alpha\beta} \cdot ~^{j_{\alpha\beta}}\big(k_\beta\big) \,\,\, .
\end{equation}
We are now ready to present a classification of
$P$-twisted bundles in terms of \v{C}ech-cohomology. 

Therefore we define the relevant cohomology set:

\begin{Definition}
Let $\big\{U_{\alpha}\big\}$ be a contractible cover of $M$ and $(j_{\alpha\beta})$ be a \v{C}ech-cocycle with values in $J$. \begin{itemize}
\item
A $(j_{\alpha\beta})$-twisted \v{C}ech-cocycle is given by a family 
$$ g_{\alpha\beta} : U_\alpha \cap U_\beta \to G $$
satisfying relation (\ref{cocycle}).
\item
Two such cocycles $g_{\alpha\beta}$ and $g'_{\alpha\beta}$ are cobordant if there exists a coboundary, that is a family of functions $k_a: U_\alpha \to G$ satisfying relation (\ref{coboundary}).
\item
The twisted \v{C}ech-cohomology set $\h^1_{j_{\alpha\beta}}(M,G)$ is defined as the quotient of twisted cocycles modulo coboundaries.
\end{itemize}
\end{Definition}

\begin{Warning}
It might be natural to guess that twisted \v{C}ech-cohomology
$\h^1_{j_{\alpha\beta}}(M,G)$ agrees with the preimage of 
the class $[j_{\alpha\beta}]$ under the map
$ \pi_*: \h^1 (M,H) \to \h^1(M,J)$. This turns out
to be wrong: The natural map
$$\begin{array}{rll}
\h^1_{j_{\alpha\beta}}(M,G) &\to& \h^1(M,H) \\
{}[g_{\alpha\beta}] &\mapsto&
[(g_{\alpha,\beta},j_{\alpha\beta})] \,\,\, ,
\end{array}$$  
is, in general, not injective.
The image of this map is always the fiber 
${\pi_*}^{-1}[j_{\alpha\beta}]$.
\end{Warning}

We summarize our findings:

\begin{Proposition}
Let $P$ be a $J$-cover of $M$, described by the cocycle 
$j_{\alpha\beta}$ over the contractible open cover
$\big\{U_{\alpha}\big\}$. Then there is a canonical bijection
$$\h^1_{j_{\alpha\beta}}(M,G) \cong \bigset{6.0cm}{Isomorphism classes of $P$-twisted $G$-bundles over $M$}.$$
\end{Proposition}

\subsection{Character theory for action groupoids}\label{groupoidrep}

In this subsection, we explicitly work out
a character theory for finite action groupoids $M//G$;
in the case of $M = pt $, this theory specializes
to the character theory of a finite group (cf. 
\cite{isaacs1994character} and \cite{serre1977linear}). 
In the special case of a
finite action groupoid coming from a finite crossed module, 
a character theory including orthogonality relation 
has been presented in \cite{bantay}. 
In the sequel, let $\KK$ be a field and denote by $\vect_{\KK}(M//G)$ the category of $\KK$-linear representations of $M//G$.

\medskip
\begin{Definition}
Let $((V_m)_{m\in M}, (\rho(g))_{g\in G})$ be a $\KK$-
linear representation of the action groupoid $M//G$ and 
denote by $P(m)$ the projection of 
$V = \bigoplus_{n \in M} V_n$ to the homogeneous component
$V_m$. We call the function
\begin{align*}
\chi: M \times G & \rightarrow \KK\\[.2em]
\chi(m,g) & := \mathrm{Tr}_V(\rho(g)P(m))
\end{align*}
the \emph{character of the representation}.
\end{Definition}

\begin{Example}
On the $\KK$-vector space
$H := \KK(M)\otimes \KK[G]$ with canonical basis 
$(\delta_m \otimes g)_{m \in M, g \in G}$,
we define a grading by $H_m = \bigoplus_{g} 
\KK(\delta_{g.m} \otimes g)$ and a group action by 
$\rho(g)(\delta_m \otimes h) = \delta_m \otimes gh$. 
This defines an object in $\vect_{\KK}(M//G)$,
called the regular representation. The character is
easily calculated in the canonical basis and found to be
\begin{align*}
\chi_H(m,g) = \sum_{(n,h) \in M\times G} \delta(g,1)\delta(h.m,n) = \delta(g,1)|G|
\end{align*}
\end{Example}

\begin{Definition}
We call a function
\[f:M \times G \rightarrow \KK \]
an \emph{action groupoid class function} on $M//G$, if it 
satisfies
\begin{center}
$
f(m,g)  = 0 \text{ if } g.m \ne m \quad\text{ and }\quad
f(h.m,hgh^{-1})  = f(m,g) \,\,. 
$
\end{center}
\end{Definition}

The character of any finite dimensional representation
is a class function.

From now on, we assume that the characteristic of
$\KK$ does not divide the order $|G|$ of the group $G$.
This assumption allows us to consider the following
normalized non-degenerate symmetric bilinear form 
\begin{align}
\langle f,f'\rangle := \frac{1}{|G|} \sum_{g \in G, m \in M} {f(m,g^{-1})} f'(m,g).\label{form}
\end{align}

In the case of {\em complex} representations, one can show,
precisely as in the case of groups,
the equality $\chi(m,g^{-1}) = \overline{\chi(m,g)}$
which allows introduce the hermitian scalar product
\begin{align}
(\chi,\chi') :=\frac{1}{|G|} \sum_{g \in G, m \in M} 
\overline{\chi(m,g)} \chi'(m,g) \,\, .\label{scalarpr}
\end{align}

\begin{Lemma} \label{orth}
Let $\KK$ be algebraically closed.
The characters of irreducible $M//G$-representations are  orthogonal and of unit length with respect to the bilinear form (\ref{form}).
\end{Lemma}

\begin{proof}
The proof proceeds as in the case of finite groups:
for a linear map $f:\ V\to W$ on the vector spaces
underlying two irreducible representations, one considers
the intertwiner
\begin{align}
f^0 = \frac{1}{|G|}\sum_{g\in G, m\in M}\rho_W(g^{-1})P_W(m)fP_V(m)\rho_V(g)\label{f0}.
\end{align}
and applies Schur's lemma.
\end{proof}\medskip

A second orthogonality relation
\[\sum_{i \in I} \chi_i(m,g)\chi_i (n,h^{-1}) = \sum_{z \in G} \delta(n,z.m) \delta(h,zgz^{-1})\]
can be derived as in the case of finite groups, as well.

Combining the orthogonality relations with the
explicit form for the character of the regular
representation, we derive in the case of
an algebraically closed field whose characteristic
does not divide the order $|G|$ use a standard reasoning:

\begin{Lemma}\label{multiplicity}
Every irreducible representation $V_i$ is contained in the 
regular representation with multiplicity $d_i:=\dim_\KK V_i$.
\end{Lemma}

As a consequence, the following generalization of
Burnside's Theorem holds:
\begin{Proposition}
Denote by $(V_i)_{i\in I}$ a set of representatives for the
isomorphism classes of simple representations of the
action groupoid and by $d_i:=\dim_\KK V_i$ the dimension
of the simple object. Then
\[\sum_{i \in I} |d_i|^2 = |M||G|\]
\end{Proposition}
\begin{proof}
One combines the relation $\dim H = \sum_{i \in I}d_i\dim V_i$
from Lemma \ref{multiplicity} with
the relation $\dim H = |M||G|$.
\end{proof}

In complete analogy to the case of finite groups, one then
shows:

\begin{Proposition}
The irreducible characters of $M//G$ form an orthogonal 
basis of the space of class functions with respect to the 
scalar product 
(\ref{form}).
\end{Proposition}

The above proposition allows us to count the number of 
irreducible representations. On the set
\[ A := \{(m,g)|g.m = m\} \subset M \times G\]
the group $G$ naturally acts by 
$h.(m,g) := (h.m,hgh^{-1})$. A class function of $M//G$ 
is constant on $G$-orbits of $A$; it vanishes
on the complement of $A$ in $M\times G$.
We conclude that the number of irreducible characters 
equals the number of $G$-orbits of $A$.

This can be rephrased as follows:
the set $A$ is equal to the set of objects of 
the inertia groupoid $\Lambda (M//G):=[\bullet//\Z,M//G]$.
Thus the number of $G$-orbits of $A$ equals the number of 
isomorphism classes of objects in $\Lambda (M//G)$, thus
$|I| = |Iso(\Lambda (M//G))|$.

\bibliographystyle{alpha}
\bibliography{Diwi}{}

\end{document}